\documentclass[a4,11pt]{amsart}

\usepackage{amssymb}
\usepackage{stmaryrd}
\usepackage[all]{xy}
\usepackage{mathrsfs}
\usepackage{color}
\usepackage{geometry}
\usepackage{enumitem}

\numberwithin{itemcounter}{subsection}

\theoremstyle{plain}

\newtheorem*{thmA}{Theorem A}
\newtheorem*{thmB}{Theorem B}

\newtheorem*{Conj}{Conjecture}

\newtheorem{theorem}{Theorem}[section]
\newtheorem{lemma}[theorem]{Lemma}
\newtheorem{lemma-definition}[theorem]{Lemma-Definition}
\newtheorem{definition-lemma}[theorem]{Definition-Lemma}

\newtheorem{proposition}[theorem]{Proposition}

\theoremstyle{definition}
\newtheorem{definition}[theorem]{Definition}
\theoremstyle{remark}
\newtheorem{remark}[theorem]{Remark}

\numberwithin{equation}{section}

\def\bbA{\mathbb{A}}

\def\bbC{\mathbb{C}}

\def\bbG{\mathbb{G}}

\def\bbN{\mathbb{N}}

\def\bbP{\mathbb{P}}
\def\bbQ{\mathbb{Q}}

\def\bbY{\mathbb{Y}}
\def\bbZ{\mathbb{Z}}

\def\scrC{\mathscr{C}}

\def\frakg{\mathfrak{g}}

\def\frakC{\mathfrak{C}}

\def\frakh{\mathfrak{H}}

\def\frakL{\mathfrak{L}}
\def\frakM{\mathfrak{M}}

\def\frakp{\mathfrak{P}}

\def\frakR{\mathfrak{R}}

\def\frakZ{\mathfrak{Z}}
\def\fraktop{{{\mathfrak{top}}}}

\def\calU{\mathcal{U}}
\def\calV{\mathcal{V}}
\def\calW{\mathcal{W}}
\def\calX{\mathcal{X}}
\def\calY{\mathcal{Y}}

\def\frakL{\mathfrak{L}}

\def\frakg{\mathfrak{g}}
\def\frakh{\mathfrak{h}}
\def\frakH{\mathfrak{H}}

\def\frakp{\mathfrak{p}}

\def\frakz{\mathfrak{z}}

\def\frakeh{\mathfrak{gh}}
\def\frakep{\mathfrak{gp}}

\def\bfr{\mathbf{r}}

\def\bfY{\mathbf{Y}}

\def\TT{{{T}}}

\def\mo{{{\bf g}}}

\def\Irr{{\operatorname{Irr}\nolimits}}

\def\sp{{\operatorname{sp}\nolimits}}

\def\e{{\operatorname{e}\nolimits}}
\def\x{{{\operatorname{x}\nolimits}}}
\def\op{{\operatorname{op}\nolimits}}
\def\ch{{\operatorname{ch}\nolimits}}

\def\G{{\operatorname{G}\nolimits}}

\def\k{{\operatorname{k}\nolimits}}

\def\Rep{{\operatorname{Rep}\nolimits}}

\def\GH{\operatorname{GH}\nolimits}
\def\GP{\operatorname{GP}\nolimits}
\def\E{\operatorname{E}\nolimits}
\def\T{\operatorname{T}\nolimits}
\def\H{\operatorname{H}\nolimits}

\def\L{\operatorname{L}\nolimits}
\def\M{{\operatorname{M}\nolimits}}
\def\P{\operatorname{P}\nolimits}

\def\R{\operatorname{R}\nolimits}
\def\SS{\operatorname{S}\nolimits}

\def\Y{\operatorname{Y}\nolimits}

\def\top{\operatorname{top}\nolimits}
\def\mid{{{mid}}}

\def\tr{{\operatorname{tr}\nolimits}}

\def\rk{{\operatorname{rk}\nolimits}}

\def\codim{{\operatorname{codim}\nolimits}}

\def\Ker{{\operatorname{Ker}\nolimits}}
\def\eu{{\operatorname{eu}\nolimits}}

\def\Im{{\operatorname{Im}\nolimits}}

\def\Hom{\operatorname{Hom}\nolimits}

\def\End{\operatorname{End}\nolimits}
\def\Res{\operatorname{Res}\nolimits}
\def\res{\operatorname{res}\nolimits}

\def\Exp{\operatorname{Exp}\nolimits}
\def\stab{\operatorname{stab}\nolimits}
\def\Ext{\operatorname{Ext}\nolimits}

\def\hyp{{\operatorname{h}\nolimits}}
\def\el{{\operatorname{e}\nolimits}}

\def\re{{\operatorname{r}\nolimits}}

\def\dil{{\operatorname{dil}}}
\def\sp{{\operatorname{sp}}}

\author{O. Schiffmann, E. Vasserot}

\title{On cohomological Hall algebras of quivers : Yangians}

\begin{document}

\maketitle

\begin{abstract}
We consider the cohomological Hall algebra $\Y^1$ of a Lagrangian substack of the moduli stack of representations of the preprojective algebra of an arbitrary quiver $Q$, 
and its actions on the cohomology of quiver varieties. 
We conjecture that $\Y^1$ is equal, after a suitable extension of scalars,
to the Yangian $\bbY$ introduced by Maulik and Okounkov, and we construct an embedding $\Y^1 \subseteq \bbY$, intertwining the respective actions of 
$\Y^1$ and $\bbY$ on the cohomology of quiver varieties.

\end{abstract}

\setcounter{tocdepth}{2}

\tableofcontents

\section{Introduction}

\medskip

Nakajima associated to each quiver $Q$ and pair of dimension vectors $(v,w)$ of $Q$ a symplectic resolution
$$\pi : \frakM(v,w) \to \frakM_0(v,w)$$
where $\frakM(v,w)$ is a smooth quasi-projective symplectic variety and $\frakM_0(v,w)$ is a (in general 
singular) affine variety.
The varieties $\frakM(v,w)$ and $\frakM_0(v,w)$ have many remarkable geometric properties and have played a very important role in geometric representation theory in the past 
twenty years (see e.g. the introduction to \cite{SV17a}). In particular, when the quiver $Q$ carries no edge loops and thus can be regarded as an orientation of the generalized Dynkin 
diagram of a Kac-Moody algebra $\mathfrak{g}_Q$, Nakajima 
constructed an action of $\mathfrak{g}_Q$ on the space 
$$L_w=\bigoplus_vH_{\top}(\frakL(v,w)),$$
where $\frakL(v,w)=\pi^{-1}(0)$ is the Lagrangian quiver variety, see \cite{N98}. The resulting module is identified with the 
integrable irreducible highest weight module of highest weight $\sum_i w_i \Lambda_i$ where the $\Lambda_i$'s are the fundamental weights of $\mathfrak{g}_Q$. When the quiver $Q$ is of \textit{finite type}, Nakajima constructed a 
representation of the quantum affine algebra of $\mathfrak{g}_Q$ on the space $$\bigoplus_{v} K^{G(w)\times\bbC^\times}(\frakL(v,w)).$$ The resulting module is called a universal \textit{standard} module, and it is a geometric analog of the global Weyl modules. 
A cohomological version of this construction, 
due to Varagnolo \cite{V}, yields an action of the Yangian of $\mathfrak{g}_Q$ on the space 
$$\bigoplus_{v} H^{G(w)\times\bbC^\times}_*(\frakL(v,w)),$$
and the resulting module is again the universal standard module.
 Finding a similar interpretation for arbitrary quivers is harder. Nakajima's and Varagnolo's actions do extend to the case of arbitrary quivers with no edge loops, but the precise nature 
 of the algebra which acts or the structure of the resulting module are not well understood beyond the cases of finite, affine or Jordan quivers. 

\smallskip

There are two main approaches to the problem of constructing and understanding symmetry algebras acting on the cohomology
of Nakajima quiver varieties for general quivers. One approach is due to Maulik and Okounkov \cite{MO}, who construct, using the RTT formalism and ideas from symplectic geometry,
 an algebra $\mathbb{Y}$, called a \emph{Yangian}, which acts on the Borel-Moore homology groups
$$F_w=\bigoplus_{v} H^{G(w)\times T}_*(\frakM(v,w))$$
for any quiver $Q$ and dimension vector $w$. Here $T$ is a natural torus acting by rescaling the edges of $Q$.
A different  approach is developed in the series of papers \cite{SV12}, \cite{SV13a}, \cite{SV13b} in the particular case of quivers with only one vertex
(and partially extended in \cite{SV17a}, \cite{YZ14}, \cite{YZ16} to the case of arbitrary quivers), in which we defined another algebra $\Y$,
called the \textit{Cohomological Hall algebra} or the \emph{K-theoretic Hall algebra} of the quiver $Q$,
acting on $F_w$ or it's K-theoretical analogue.
This construction is based on the geometry of the cotangent of the moduli stacks of representations of quivers. 
The algebra $\Y$ is in some sense the largest algebra which acts on the homology of all quiver varieties by means of some \emph{Hecke correspondences}.
 
\smallskip

The aim of this paper is to compare the algebras $\mathbb{Y}$ and $\Y$ as well as their respective actions on the spaces $F_w$ for an arbitrary quiver $Q$ and dimension vector $w$,  
see Theorem~B below. In order to state our results more precisely, we need to introduce some notations and recall some results from \cite{SV17a} and \cite{MO}.

\smallskip

\emph{Cohomological Hall algebras.} Let $v$ be a dimension vector and let $Rep (\mathbb{C}Q,v)/ G(v)$ be the moduli stack of complex representations of $Q$ of dimension $v$. 
Let $\Pi$ be the preprojective 
algebra of $Q$ and  $Rep(\Pi, v)/G(v)$ be the moduli stack of complex representations of $\pi$ of dimension $v$. 
In \cite{BSV} we defined a Lagrangian substack $\Lambda^1(v)/G(v)$ of 
$Rep(\Pi, v)/G(v)$, by using some semi-nilpotency condition. See also \cite{B14}. 
The torus $T$ acts on 
$Rep(\Pi,v)/G(v)$ and $\Lambda^1(v)/G(v)$. We refer to \cite[thm~A]{SV17a} for a list of geometric properties of $\Lambda^1(v)/G(v)$. Set
$$\Y^{1}=\bigoplus_v H_*^{G(v)\times T}(\Lambda^1(v)), \qquad \Y=\bigoplus_v H_*^{G(v)\times T}(Rep(\Pi,v)).$$
Motivated by the analogy with Yangians, 
we consider an extension $\bfY^1$ of $\Y^1$ by adding a \emph{loop Cartan} part equal to 
$\bfY(0)=H^*_{G(\infty)\times T}.$
Let $\Bbbk$ be the $T$-equivariant cohomology ring $H^*_T$ of the point and $K$ be its fraction field.
We write
$$\Y_K^1=\Y^1\otimes_{\Bbbk}K,\quad \bfY^1_K=\bfY^1 \otimes_{\Bbbk} K.$$
We set 
$$\Lambda_{(v)} =\{0\} \times Rep(\mathbb{C}Q^*,v).$$ 
This is always an irreducible component of $\Lambda^1(v)$ in case $v$ is supported on a subquiver without oriented cycles. 
We say that an imaginary vertex is \textit{elliptic}
or \textit{isotropic} if it carries a single $1$-loop and \textit{hyperbolic} if it carries more than one $1$-loop. We denote by 
$I^\re,$ $ I^\el$ and $ I^\hyp$ respectively the real, elliptic and hyperbolic vertices of $I$.
Let  $\{\delta_i\,;\,i \in I\}$ be the basis of delta functions in $\bbZ^I$.
The important properties of $\Y,$ $\Y^1$ and $\mathbf{Y}$ proved in \cite{SV17a} are summarized below.

\smallskip

\begin{thmA}
For any $\#=\emptyset,1,$ we have the following.
\begin{enumerate}
\item[$\mathrm{(a)}$] There is an associative $\mathbb{Z} \times \mathbb{N}^I$-graded $\Bbbk$-algebra structure on $\Y^{\#}$.
\item[$\mathrm{(b)}$] There is a representation of $\Y^\#$ on $F_w$ for each $w$.
\item[$\mathrm{(c)}$] 
The diagonal action of $\Y^1$ on $\bigoplus_w F_w$ is faithful.
\item[$\mathrm{(d)}$] There are $K$-algebra isomorphisms
$\Y^1_K  \simeq \Y_K.$
\item[$\mathrm{(e)}$] The $K$-algebra $\Y_K$ is generated by the subspaces 
$H^{G(v)\times T}(\Lambda_{(v)})\otimes_\Bbbk K$ where $v$ ranges among the following set of dimension vectors 
$$\{\delta_i\;;\; i \in I^\re \cup I^\el\} \cup \{l\delta_i\;;\; i \in I^\hyp, l \in \mathbb{N}\}.$$
The $K$-algebra $\bfY_K$ is generated by $\bfY(0)$ and the collection of fundamental classes
$$\{[\Lambda_{(\delta_i)}]\;;\; i \in I^\re \cup I^\el\} \cup \{[\Lambda_{(l\delta_i)}]\;;\; i \in I^\hyp, l \in \mathbb{N}\}.$$
\end{enumerate}
\end{thmA}

\smallskip

This implies that $\Y^1$ coincides with the $\Bbbk$-algebra constructed by Varagnolo when $Q$ has no imaginary vertex. 
If $Q$ has no hyperbolic vertex then the same is true after extension of scalars to $K$.

\smallskip

\emph{Maulik-Okounkov Yangians.}
In \cite{MO}, the authors defined and studied another associative algebra acting on the Borel-Moore
homology of Nakajima quiver varieties associated to an arbitrary quiver $Q$. 
Their construction, which stems from ideas in symplectic 
geometry, hinges on the notion of \textit{stable envelope} to produce a quantum $R$-matrix, and then on the RTT formalism to define an associative $\mathbb{Z}\times \mathbb{Z}^I$-graded $\Bbbk$-algebra $\mathbb{Y}=\mathbb{Y}_Q$ acting on the space $F_w$ for 
any dimension vector $w$. Taking a quasi-classical limit, they also define a classical $R$-matrix and a $\mathbb{Z} \times \mathbb{Z}^I$-graded Lie algebra $\mo=\mo_Q$. If $Q$ is of finite type, then 
$\mo_Q$ is the semisimple Lie algebra $\mathfrak{g}_Q$ associated with $Q$, and $\mathbb{Y}_{Q}$ is the Yangian of the same type. 
In general, the $\Bbbk$-algebra $\mathbb{Y}_Q$ is a deformation of the enveloping algebra of the current algebra 
$\mo[u]$.
Their construction provides triangular decompositions 
 $$\mathbb{Y}=\mathbb{Y}_+ \otimes \mathbb{Y}_0 \otimes \mathbb{Y}_-, \qquad \mo=\mo_{+} \oplus \mo_{0} \oplus \mo_{-}.$$

\smallskip

\emph{Main result and conjecture.} In this paper, we compare $\Y$ with the positive half $\mathbb{Y}_+$. 
Set $\bbY_{+,K}=\mathbb{Y}_{+}\otimes_{\Bbbk} K$. We make the following conjecture in Remark \ref{conjecture}.

\smallskip

\begin{Conj} There is a unique $K$-algebra isomorphism $\mathbb{Y}_{+,K} \simeq \Y_K$ (up to some central elements)
which intertwines the actions of $\mathbb{Y}_{+,K}$ and $\Y_K$ on $F_w \otimes_{\Bbbk} K$ for any $w$.
\end{Conj}

\smallskip

The main result of this paper is one half of the above conjecture. It is proved in Theorem~\ref{T:main}.

\smallskip

\begin{thmB}\label{T:C} There is a unique embedding $ \Y_K \subset \mathbb{Y}_{+,K}$ which intertwines the
actions of $\mathbb{Y}_{+,K}$ and $\Y_K$ on $F_w \otimes_{\Bbbk} K$ for any $w$.
\end{thmB}

\smallskip

By Theorem~A(e), the proof of the above theorem boils down to checking that certain \textit{generalized Hecke correspondences} 
corresponding to generators of $\Y^1$ occur with a  non zero and constant coefficient in a suitable stable envelope. Again, 
generalized Hecke correspondences associated to real, elliptic or hyperbolic vertices behave in very different ways and we have 
to treat each case separately. The case of real vertices was already considered in \cite{McB}.

\smallskip

\emph{Examples of COHAs and Yangians.}
 Here we collect the few instances for which the algebras $\mathbb{Y}$ and $\Y$ are explicitly known.

\smallskip

Assume first that $Q$ is a quiver of finite type, and let $\mathfrak{g}$ be the corresponding finite dimensional simple Lie algebra. The $r$-matrix used to define the Yangian 
$\mathbb{Y}$ was explicitly calculated in \cite{McB}. It was shown there that $\mathbb{Y}_K$ coincides with the positive half of Varagnolo's algebra which is the 
standard Yangian $Y^+(\mathfrak{g})$ of $\frakg$. 
The $K$-algebra $\Y^1_K$ coincides with Varagnolo's algebra as well, and thus our main conjecture 
is true here. Note also that since $Rep(\Pi,v)=\Lambda^1(v)$ we have $\Y^1=\Y$.

\smallskip

Now, let $Q$ be the Jordan quiver. In that case the $r$-matrix has been calculated in \cite{MO} and the Yangian $\mathbb{Y}$ is identified with the  positive half of the
\emph{affine Yangian of $\mathfrak{gl}(1)$} introduced in \cite{SV13b}. This algebra was defined as the subalgebra of $\Y$ generated by 
$\mathbf{Y}(0)$ and $H_*^{G(1)\times T}(\Rep(\Pi,1))$.
By Theorem~A(e) above, it is equal to the whole of $\Y$. 
Hence in this case also our conjecture is verified. This is the simplest example for which $\Y^1 \neq \Y$.

\smallskip

Next, let $Q$ be an affine quiver and $\mathfrak{g}$ be the associated affine Kac-Moody algebra. 
As far as we know, the Yangian 
$\mathbb{Y}$ has not been fully described in this case. As there are no imaginary vertices, 
the algebra $\Y^1$ coincides with Varagnolo's algebra, which is known to admit a surjective map from 
the Yangian $Y^+(\mathfrak{g})$. Comparing graded dimensions using \cite[thm~A(c)]{SV17a} one should be able to check that $\Y^1$ is in fact equal to $Y^+(\mathfrak{g})$. 
The Kac-Moody algebra $\frakg$ is the central extension of the loop algebra of a finite dimensional simple Lie algebra $\frakg_0$. Then, the Yangian
$Y(\mathfrak{g})$ is a deformation of of the enveloping algebra of
$$\mathfrak{g}_0[t, s^{\pm 1}] \oplus \bigoplus_{l \in \mathbb{Z}, m \geq 0} \mathbb{C}\,c_{l,m}$$
where the central element $c_{l,m}$ belongs to the imaginary root space $l\delta_s + m\delta_t$. In this case the algebras
$\Y$ and $\mathbb{Y_+}$ are both strictly bigger than the positive half of $U(\mathfrak{g}[u])$.

\smallskip

Not much is known for wild quivers, except for the fact that when $Q$ has no imaginary vertex there is always a surjective map from the Yangian $Y^+(\mathfrak{g})$ of the Kac-Moody 
algebra $\mathfrak{g}$ to $\Y_K$, and that $\Y_K$ is strictly larger than $U(\mathfrak{g}[t])$.
This last fact comes from the existence of imaginary roots and hence non-constant Kac polynomials.

\smallskip

\emph{Okounkov's conjecture.} Let us finish this introduction by mentioning another motivation for this work. In \cite{O13}, A. Okounkov conjectured that
the graded dimensions of the root spaces $\mo[\alpha]$ of $\mo$ are, after a suitable grading shift, precisely given by the Kac 
polynomial $A_{\alpha}(t)$. This is a vast generalization of the famous Kac conjecture, proved in
\cite{Ha10}, stating that for quivers with no $1$-cycles the multiplicity of the root $\alpha$ in the Kac-Moody algebra 
$\mathfrak{g}_Q$ is equal to the constant term $A_{\alpha}(0)$. Indeed, one expects to have an isomorphism 
$\mathfrak{g}_Q \simeq\mo_Q[0]$, where the grading in $\mo_Q$ is counted from middle dimension up.
In other words, the Lie algebra $\mo_Q$ should be a graded extension of $\mathfrak{g}_Q$, whose character is conjecturally 
given by the \textit{full} Kac polynomials rather than their constant terms. As proved in \cite{SV17a},  The Poincar\'e polynomials $P(\Lambda^1(v),q)=\sum_{i}  \dim_{\mathbb{Q}}(H^{G(v)\times T}_{2i}(\Lambda^1(v)))\,q^{i}$ of $\Lambda^1(v)$
 are determined by the formula
\begin{equation}\label{E:Okounkovconj}
\sum_v P(\Lambda^1(v),q)\,q^{\langle v, v \rangle + \rk(T)}z^v= 
(1-q^{-1})^{-\rk(T)}\Exp\Big(\sum_v \frac{A^1(q^{-1})}{1-q^{-1}}z^v\Big)
\end{equation}
where $A^1(t)$ is the nilpotent Kac polynomial and $\Exp$ is the plethystic exponential. 
This implies a variant of the above 
conjecture~: the graded character of the cohomological Hall algebras $\Y^1$ are encoded by the full \textit{nilpotent} Kac 
polynomials $ A^1_{\alpha}(t)$. A formula similar to (\ref{E:Okounkovconj}) but with $Rep(\Pi,v)$ in place of $\Lambda^1(v)$ 
can be deduced from \cite{Moz} and the purity of $Rep(\Pi,v)$ proved in \cite{Davisonpurity}; it involves the \textit{usual} Kac polynomial $A(q)$.
Note that our construction does not directly provide a construction of an underlying Lie algebra. 
See some recent work of Davison and Meinhardt in that direction \cite{DM}. 
We expect the discrepancy between the various types of Kac polynomials involved to correspond to different gradings 
on the \textit{same} Lie algebra. Indeed, the $\Bbbk$-algebras $\Y$  and $\Y^1$ are different integral forms inside 
the $K$-algebra $\Y_K$. Note also that we have $A_{\alpha}(1)=A^1_{\alpha}(1)$.

\smallskip

\emph{Plan of the paper.} Section 2 contains some recollections on quiver varieties and some results of \cite{SV17a} pertaining to the geometry of the stacks $\Lambda^1(v)$. Section~3 provides the definitions of Maulik-Okounkov's Yangians and of the cohomological Hall algebras. Our main theorem is proved in Section~4. This is done by explicitly realizing the action on $F_w$ of the generators  for $\Y^1$ given in \cite[thm B]{SV17a} as some convolution operators with explicit generalized Hecke correspondences, 
 by showing that these Hecke correspondences occur in the stable envelope of \cite{MO}, and, finally, by using the faithfulness of the action of $\Y^1$ on $\bigoplus_w F_w $.

\bigskip

\medskip

\centerline{\bf{\large{Acknowledgements}}}

\smallskip

We would like to thank  M. McBreen, H. Nakajima, A. Negut and A.~Okounkov for useful discussions and correspondences. 
\newpage

\medskip

\section{Reminder on quiver varieties}

We'll use the same notation as in \cite{SV17a}, to which we refer for more details.
To facilitate the reading, we recall some of them here.

\smallskip

\subsection{Generalities}\hfill\\

All schemes considered here will be reduced.
By a \emph{variety} we'll mean a reduced scheme over the field $\k=\bbC$. 
A subvariety $X$ of a symplectic manifold $M$ is \emph{isotropic} 
if the restriction of the symplectic form to 
the smooth locus of $X$ vanishes. Let $M^{\op}$ denote the manifold $M$ with the opposit symplectic form.

\smallskip

Given an algebraic variety $X$, let
$H_*(X)$ be the Borel-Moore homology group (= the locally finite singular homology group) of $X$.
Now, assume that $X$ is equipped with the action of a linear algebraic group $G$.
We'll always assume that $X$ is quasi-projective and that we have fixed a very ample $G$-linearized
line bundle over $X$.
Thus, the variety $X$ embeds equivariantly in a projective space with a linear $G$-action. 
Let $H_*^G(X)$ and $H^*_G(X)$
be the equivariant Borel-Moore homology and the equivariant cohomology group of $X$ with rational coefficients.
We'll abbreviate $H^*_G=H^*_G(\bullet)$.

\smallskip

The \emph{quotient stack} of a $G$-variety $X$ by $G$ is the Artin stack associated with
the groupoid $G\times X\to X\times X$.  We'll denote it by the symbol $X/G$. We'll use
the symbol $[X/G]$ for the fundamental class of $X/G$ and the symbol $X/\!\!/G$ for the categorical quotient.

\smallskip

\subsection{Quivers}\hfill\\

Let $Q=(I,\Omega)$ be a finite quiver with set of vertices $I$ and set of arrows $\Omega$.
Let $Q^*=(I,\Omega^*)$ be the opposite quiver, with the set of arrows $\Omega^*=\{h^*\,;\,h\in \Omega\}$
where $h^*$ is the arrows obtained by reversing the orientation of $h$.
The double quiver is $\bar Q=(I,\bar\Omega)$ with $\bar\Omega=\Omega\sqcup\Omega^*$.
We set $\varepsilon(h)=1$ if $h\in \Omega$ and $-1$ if $h\in \Omega^*$.
Let $h'$, $h''$ be the source and the goal in $I$ of an arrow $h\in\Omega$.
For each vertex $i$ let 
$\Omega_{ij}$ be the set of arrows in $\Omega$ from $i$ to $j$. 
Put
$\bar\Omega_{ij}=\Omega_{ij}\cup(\Omega_{ji})^*$.
If $i\neq j$ we write
$$q_i=|\Omega_{ii}|,\quad q_{ij}=|\Omega_{ij}|,\quad \bar q_{ij}=|\bar\Omega_{ij}|,\quad q=(q_{ij}\,;\,i,j\in I),\quad $$
Fix a tuple $v=(v_i\,;\,i\in I)$ in $\bbZ^I$.
The Ringel bilinear form $\langle \bullet,\bullet\rangle$ on $\bbZ^I$ is given by
$$\langle v,w\rangle=v\cdot w-\sum_{h\in\Omega}v_{h'}w_{h''},\quad v\cdot w=\sum_{i\in I}v_iw_i.$$
Let $(v,w)=\langle v,w\rangle+\langle w,v\rangle$ be the Euler bilinear form. To avoid confusions we may write
$(v,w)_Q$ for $(v,w)$. 
For each $i\in I$ let $\delta_i\in\mathbb{Z}^I$ denote the delta function at the vertex $i$.
Set
\begin{align*}d_v=v\cdot v-(v,v)/2.\end{align*}

\smallskip

Let $\k Q$ be the path algebra of $Q$.
A \emph{dimension vector} of $Q$ is a tuple $v\in\bbN^I$.
Let $\k^v$ denote the $I$-graded vector space $\k^v=\bigoplus_{i\in I}\k^{v_i}$.
We may abbreviate 
$V=\k^v$ and $V_i=\k^{v_i}.$
Let $Rep(\k Q,v)$ be the set of representations of $\k Q$ in $V$, with its natural structure of affine $\k$-variety.
Consider the symplectic vector space
$$\R(v)=Rep(\k \bar Q,v)=Rep(\k Q,v)\oplus Rep(\k Q^*,v).$$
An element of $\R(v)$ is a pair $\bar x=(x,x^*)$ where $x=(x_h\,;\,h\in\Omega)$ belongs to $Rep(\k Q,v)$ and 
$x^*=(x_h\,;\,h\in\Omega^*)$ belongs to $Rep(\k Q^*,v)$.
The algebraic group $G(v)=\prod_{i\in I}GL(v_i)$ acts by conjugation on $\R(v)$, preserving the symplectic form.
Let $\frakg(v)$ be the Lie algebra of $G(v)$ and
$\mu\,:\,\R(v)\to\frakg(v)$ be the moment map, which is given by
$$\mu(\bar x)=[x,x^*]=\sum_{h\in\Omega} [x_{h},x_{h^*}].$$
We define
$\M(v)=\mu^{-1}(0).$
The preprojective algebra is defined as
$$\Pi= \k \bar{Q} / \langle \mu(\bar x)\,;\,\bar x\in\R(v)\rangle.$$

\smallskip

Consider the group
\begin{align*}
G_\Omega=\prod_{i\in I}SP(2q_i)\times\prod_{i\neq j}GL(q_{ij})\times G_\dil.
\end{align*}
The second product is over all pairs $(i,j)$ in $I\times I$ such that $i\neq j$ and $G_\dil=\bbG_m$.
The group $G_\Omega$ acts on $\R(v)$ so that
the factor $G_\dil$ acts by dilatation on the summand $Rep(\k Q,v)^*$. 
Let $T\subset G_\Omega$ be a torus of the form
\begin{align*}\TT=\TT_{\sp}\times\TT_{\dil},\quad
\TT_{\sp}\subseteq(\bbG_m)^\Omega,\quad\TT_\dil\subseteq G_\dil\end{align*}
where a tuple $(z_h\,,\,z)$ acts by $z_h$ on $x_h$ and
by $z/z_h$ on $x_{h^*}$ for all $h\in\Omega$.
We'll assume that the torus $T$ contains a one parameter subgroup which scales all the quiver 
data by the same scalar, so we have $T_\dil\neq\{1\}$.
Let $\hbar$ be the character of $T$
which is trivial on  $T_\sp$ and coincides with the identity on $T_\dil$.

\smallskip

\subsection{The semi-nilpotent variety}\label{sec:sn}

\subsubsection{}
Fix an increasing flag  $W$ of $I$-graded vector spaces in $V$
$$W=(\{0\}=W_0\subsetneq W_1\subsetneq\cdots\subsetneq W_r=V).$$
Then, we consider the closed subset $\Lambda_W$ of $\M(v)$ given by
\begin{equation*}
\Lambda_{W}=\{\bar x\in \M(v)\,;\,x(W_p)\subseteq W_{p-1}\,,\,x^*(W_p)\subseteq W_{p}\}.
\end{equation*}
Up to conjugacy by an element of $G(v),$ the flag $W$ is completely determined by the sequence of dimension vectors 
\begin{align*}\nu_1=\dim(W_1/W_0),\dots,\nu_r=\dim(W_r/W_{r-1}).
\end{align*}
The tuple $\nu=(\nu_1,\dots,\nu_r)$ is a \emph{composition of $v$}, i.e., it is a tuple of dimension vectors  
with sum $v$.  We'll write $\nu\vDash v$.
We'll say that the flag $W$ is \emph{of type $\nu$}.
Then, we define
\begin{align*}\Lambda_{\nu}=G(v)\cdot\Lambda_{W}\subseteq \M(v),\end{align*}
where the dot denotes the $G(v)$-action on $\M(v)$.
We'll say that a composition $\nu\vDash v$ is \emph{restricted}
if each part is concentrated in a single vertex.
The  \emph{strongly semi-nilpotent variety} is the closed 
subset 
of $ \M(v)$ given by
\begin{align}\label{Lambda0}
\Lambda^1(v)=\bigcup_{\nu}\Lambda_{\nu},
\end{align}
where $\nu$ runs over the set of all restricted compositions of $v$.
The $G(v)\times T$-action on $\M(v)$ yields a $G(v)\times T$-action on $\Lambda^1(v)$.
In \cite[thm.~3.2]{SV17a} we have proved the following.

\begin{theorem}\label{thm:1.1} \hfill 
\begin{itemize}[leftmargin=8mm]

\item[$\mathrm{(a)}$] $\Lambda^1(v)$ is a closed Lagrangian subvariety of $\R(v)$, of dimension $d_v$.

\item[$\mathrm{(b)}$] $H_*^{G(v)\times T}(\Lambda^1(v))$ is pure and even.



\item[$\mathrm{(c)}$] $H_*^{G(v)\times T}(\Lambda^1(v))$ is free as an $H^*_{T}$-module.
\end{itemize}
\end{theorem}

\smallskip

\begin{remark}\label{rem:ind}
Let $(v)$ be the composition of $v$ with only one term.
Then, we have
$$\Lambda_{(v)}=\{0\}\times Rep(\k Q^*,v).$$
\end{remark}

\smallskip

\subsubsection{}
Let us consider in more details the case of the quiver 
with vertex set $\{i\}$ and $q_i$ loops. Then, the dimension vector $v$ is an integer.

\smallskip

First, assume that $q_i=1$.
Then $Q$ is the Jordan quiver and $\Lambda^1(v)$ is the set of 
commuting pairs $(x,x^*)$ in $\frakg(v)\times\frakg(v)$ with $x$ nilpotent.
We identify $\frakg(v)\times\frakg(v)$ with $\frakg(v)\times\frakg(v)^*$ via the trace map.
Then, the irreducible components of $\Lambda^1(v)$ are the closure of the conormal bundles to the nilpotent
$G(v)$-orbits in $\frakg(v)$.

\smallskip

Next, assume that $q_i>1$.
Then, by \cite[prop.~�3.8]{SV17a} 
the variety $\Lambda_{\nu}$ is irreducible and Lagrangian in $\R(v)$ for each $\nu$ and
the set of irreducible components of $\Lambda^1(v)$ is
\begin{align}\label{IrrLambda}\Irr(\Lambda^1(v))=\{\Lambda_{\nu}\,;\,\nu\vDash v\}.\end{align}

\smallskip

\subsection{Quiver varieties}

\subsubsection{} 
The symplectic vector space of representations of dimension vectors $v,$ $w$ of the \emph{framed quiver} 
associated with $\bar Q$ is
\begin{align*}
\R(v,w)=\R(v)\oplus\Hom_I(w,v)\oplus\Hom_I(v,w),
\end{align*}
where $\Hom_I(v,w)$ is the set of $I$-graded $\k$-linear maps $V\to W$ and
$$V=\k^v,\quad W=\k^w,\quad V_i=\k^{v_i},\quad W_i=\k^{w_i}.$$
The algebraic group $G(v)\times G(w)\times T$ acts on $\R(v,w)$.
The symplectic form is homogeneous of weight $\hbar$ under the $G(v)\times G(w)\times T$-action.
The $G(v)$-action is symplectic and admits the
moment map 
$$\mu:\R(v,w)\to\frakg(v),\quad(\bar x,\bar a)\mapsto[x,x^*]+aa^*$$
where we write 
\begin{align*}
&\bar x=(x,x^*)\in \R(v),\\
&\bar a=(a,a^*)\in\Hom_I(w,v)\oplus\Hom_I(v,w).
\end{align*}
Let $\frakM_0(v,w)$ be the categorical quotient by $G(v)$ of the zero set 
\begin{align*}\M(v,w)=\mu^{-1}(0).\end{align*} 
It is affine, reduced, irreducible, singular in general and $G(w)\times T$-equivariant. 
Given a character $\theta$ of $G(v)$ we consider the space of semi-invariants of weight $\theta$
$$\k[\M(v,w)]^{\,\theta}\subseteq\k[\M(v,w)].$$
We have the $G(w)\times T$-equivariant projective morphism
$$\pi_\theta : \frakM_\theta(v,w)=\text{Proj}\Big(\bigoplus_{n\in\bbN}\k[\M(v,w)]^{\,\theta^n}\Bigr) \to \frakM_0(v,w).$$
The Hilbert-Mumford criterion implies that 
$\frakM_\theta(v,w)$ is the geometric quotient by $G(v)$ of an open subset $\M_\theta(v,w)$ of $\M(v,w)$ 
consisting of the $\theta$-\emph{semistable} representations.
Replacing everywhere $\M(v,w)$ by $\R(v,w)$ we define an open subset $\R_\theta(v,w)$ of $\R(v,w)$ such that
$\M_\theta(v,w)=\R_\theta(v,w)\cap \M(v,w)$.
If $\theta$ is \emph{generic} then any semistable pair $(\bar x,\bar a)$ is stable and the variety $\frakM_{\theta}(v,w)$ is smooth, symplectic, of dimension 
\begin{equation*}
d_{v,w}=2\,v\cdot w-(v,v)_Q.
\end{equation*} 
The character $\theta=s$ given by $s(g)=\prod_{i\in I}\det(g_i)^{-1}$ is generic.
A representation $(\bar x,\bar a)$ in $\R(v,w)$ is ($s$-)semistable
if and only if it does not admit any nonzero subrepresentation
whose dimension vector belongs to $\bbN^I\times\{0\}$.
We abbreviate
\begin{align*}
\frakM(v,w)=\frakM_s(v,w),\quad \pi=\pi_s.
\end{align*}
Let $[\bar x,\bar a]$ be the image in $\frakM(v,w)$ of the tuple $(\bar x,\bar a)\in\M_s(v,w)$.

\smallskip

It may be useful to realize the 
quiver varieties as moduli spaces of representations of some preprojective algebra. 
More precisely, consider the quiver $\tilde{Q}=(\tilde{I}, \tilde{\Omega})$ obtained from $Q$ 
by adding one new vertex $\infty$ and $w_i$ arrows from $\infty$ to the vertex $i$ for all $i \in I$.
For each $v \in \mathbb{N}^I$ set 
$$\tilde{v}=v+\delta_\infty\ \text{if}\ w\neq 0,\ \text{and}\ \tilde v=v\ \text{else}.$$
If $w\neq 0$ we'll identify $G(v)$ with $PGL(\tilde{v})$. 
Then, the varieties $\frakM(v,w)$ or $\frakM_0(v,w)$ 
may be viewed as moduli spaces 
of (stable, semisimple)  modules of dimension $\tilde{v}$
over the preprojective algebra $\tilde\Pi$ of $\tilde Q$.

\smallskip

\subsubsection{} 
For any closed point $t$ of the center $\frakz(v)$ of $\frakg(v)$, the variety
\begin{align*}\M(v,w)_t=\mu^{-1}(t)\end{align*} 
gives rise to a projective morphism
\begin{align*}\pi : \frakM(v,w)_t \to \frakM_0(v,w)_t.\end{align*}
We'll call $\frakM(v,w)_t$ or $\frakM_0(v,w)_t$ a \emph{deformed} quiver variety.
Given a generic line $\bbA^1\subseteq\frakz(v),$ we may consider the families of varieties
\begin{align*}\frakM(v,w)_{\bbA^1}\to\bbA^1,\quad \frakM_0(v,w)_{\bbA^1}\to\bbA^1.\end{align*}
If $\frakM(v,w)$ is non empty, then $\frakM(v,w)_t$ and $\frakM_0(v,w)_t$ are also non empty for each $t\neq 0$ and we get a projective morphism $\pi : \frakM(v,w)_{\bbA^1} \to \frakM_0(v,w)_{\bbA^1}$
over $\bbA^1$ which is an isomorphism above ${\bbG_m}.$

\smallskip

\subsubsection{}
If the representation $z \in\M(v,w)$ is semisimple then we can decompose it into its simple constituents
$z=z_1^{\oplus d_1}\oplus z_2^{\oplus d_2}\oplus\cdots z_s^{\oplus d_s}$ where the $z_r$'s are non-isomorphic simples.
If $u_r=(v_r,w_r)$ is the dimension vector of $z_r$, i.e., if $z_r$ is a simple representation in $\M(v_r,w_r)$, 
then we say that $z$ has the 
\emph{representation type} 
$$\tau=(d_1,u_1\,;\,d_2,u_2\,;\,\dots\,;\,d_s,u_s).$$
If $w\neq 0$ and $z$ is stable,
then there is a unique  integer $r$ such that $w_r=w$ and $w_{r'}=0$ for all $r'\neq r$, hence
we may assume that the representation type of $z$ has the following form
\begin{align*}\tau=(1,v_1,w\,;\,d_2,v_2\,;\,\dots\,;\,d_s,v_s),\end{align*}
where $(v_1,v_2,\dots,v_s)\vDash v$ and the tuples
$(d_2,v_2),\dots(d_s,v_s)$ are only defined up to a permutation.
Let $RT(v,w)$ be the set of all representation types of dimension $(v,w)$. Let
$$\frakM_0(\tau)\subseteq\frakM_0(v,w)$$
be the set of semisimple representations with representation type equal to $\tau$.
Since $\frakM_0(v,w)$ is irreducible, there is a unique representation type $\kappa_{v,w}$ such that
$\frakM_0(\kappa_{v,w})$ is a dense open subset of $\frakM_0(v,w)$.
We'll call it the \emph{generic representation type} of $RT(v,w)$.
Given two representation types
\begin{align*}
\tau=(1,v_1,w\,;\,d_2,v_2\,;\,\dots\,;\,d_s,v_s)\in RT(v,w),\\
\kappa=(1,u_1,z\,;\,e_2,u_2\,;\,\dots\,;\,e_r,u_r)\in RT(u,z),
\end{align*}
we define their sum by
$$\tau\oplus\kappa=(1,v_1+u_1,w+z\,;\,d_2,v_2\,;\,\dots\,;\,d_s,v_s\,;\,e_2,u_2\,;\,\dots\,;\,e_r,u_r).$$
If $\sum_{t=1}^rv_t\leqslant v$ and 
$\sum_{t=1}^rw_t\leqslant w$,
the direct sum of quiver representations yields a closed embedding
$\oplus:\prod_{t=1}^r\frakM_0(v_t,w_t)\to\frakM_0(v,w)$
such that
$$\frakM_0(\tau)\oplus\frakM_0(\kappa)\subseteq\overline{\frakM_0(\tau\oplus\kappa)}.$$
For each representation type $\tau$ we write 
\begin{align*}
\M(\tau)&=\rho_0^{-1}(\frakM_0(\tau)),\\
\M_s(\tau)&=\M(\tau)\cap\M_s(v,w),\\
\frakM(\tau)&=\M_s(\tau)/\!\!/ G(v).
\end{align*}
If $\frakM(\tau)\neq\emptyset$ then the map $\pi$ restricts to a locally trivial fibration 
$\frakM(\tau)\to\frakM_0(\tau).$

\smallskip

\subsubsection{}\label{sec:BBQ}
The results of this section apply to arbitrary deformed quiver varieties.
To simplify the exposition we'll deal only with ordinary ones.
Let $\gamma$ be a cocharacter of $G(w)\times T$. Composing it with the $G(w)\times T$-action we get a
$\bbG_m$-action $\bullet$ on $\frakM(v,w)$. We want to describe 
the fixed point locus $\frakM(v,w)^\bullet$ and the Bialynicki-Birula \emph{attracting variety}
\begin{align*}\frakL(v,w)^\bullet=\{z\in\frakM(v,w)\,;\,\exists\lim_{t\to 0}t\bullet z\}.\end{align*}
To do this, fix a cocharacter $\rho$ of $G(v)$. 
Since the $G(v)$-action on $\M_s(v,w)$ commutes with the $G(w)\times T$-action, we can view the product $\gamma\rho$ as
a cocharacter of $G(v)\times G(w)\times T$.
Let $L$ be the centralizer of $\rho$ in $G(v)$ and set
\begin{align*}
\begin{split}
P&=\{g\in G(v)\,;\,\exists\lim_{t\to 0}\rho(t)\,g\,\rho(t)^{-1}\},\\
U&=\{g\in G(v)\,;\,\lim_{t\to 0}\rho(t)\,g\,\rho(t)^{-1}=1\},\\
\M[\rho]&=\{z\in\M_s(v,w)\,;\,\gamma\rho(t)\cdot z=z\,,\,\forall t\},\\
\L[\rho]&=\{z\in\M_s(v,w)\,;\,\lim_{t\to 0}\gamma\rho(t)\cdot z\in\M[\rho]\}.
\end{split}
\end{align*}
Then, we have
\begin{align*}
\begin{split}
\frakM[\rho]&:=\big(G(v)\cdot\M[\rho]\big)\,/\!\!/\,G(v)=\M[\rho]\,/\!\!/\,L,\\
\frakL[\rho]&:=\big(G(v)\cdot\L[\rho]\big)\,/\!\!/\,G(v)=\L[\rho]\,/\!\!/\,P.
\end{split}
\end{align*}
The set
$\frakM[\rho]$ is a sum of connected components of $\frakM(v,w)^\bullet$.
We have
$$\frakL[\rho]=\{z\in\frakM(v,w)\,;\,\lim_{t\to 0}t\bullet z\in\frakM[\rho]\}$$ and
$$\frakM(v,w)^\bullet=\bigsqcup_\rho\frakM[\rho],\quad\frakL(v,w)^\bullet=\bigsqcup_\rho\frakL[\rho].$$

\smallskip

\subsubsection{}
Set
\begin{align*}
\L_s^1(v,w)=\R_s(v,w)\cap\big(\Lambda^1(v)\times\{0\}\times\Hom_I(v,w)\big).
\end{align*}
The \emph{Lagrangian quiver variety} is the geometric quotient
\begin{align*}\frakL^1(v,w)=\L_s^1(v,w)/\!\!/G(v).\end{align*}
We have closed embeddings
\begin{align*}
\L_s^1(v,w)  \subset \M_s(v,w),\quad
\frakL^1(v,w)\subset \frakM(v,w).
\end{align*}
 The following is proved in \cite[prop. 3.1, 3.2]{BSV}.

\smallskip

\begin{proposition}\label{prop:TL}\hfill
\begin{itemize}[leftmargin=8mm]
\item[$\mathrm{(a)}$] 
There is a $\bbG_m$-action $\bullet$ on $\frakM(v,w)$ such that $\frakL^1(v,w)=\frakL(v,w)^\bullet$.
\item[$\mathrm{(c)}$] 
$\frakL^1(v,w)$ is a closed Lagrangian subvariety of $\frakM(v,w)$.
\end{itemize}
\qed
\end{proposition}

\smallskip

\subsection{Homology of quiver varieties}\hfill\\

Set $\Bbbk=H^*_{T}$ and $ \Bbbk[w]=H^*_{G(w)\times T}.$
Let $K$ be the fraction field of $\Bbbk$.
We'll abbreviate
$$\otimes=\otimes_\Bbbk,\quad\Hom=\Hom_\Bbbk,\quad (\bullet)^\vee=\Hom(\bullet,\Bbbk).$$
We may view $\hbar$ as a non zero element of $\Bbbk$ of degree 2.
Let $F_{w}$ be the $\bbZ^I\times \bbZ$-graded $\Bbbk[w]$-module given by
$$F_w=\bigoplus_{v\in\bbN^I}F_w(v),\quad F_w(v)=\bigoplus_{k\in\bbZ}F_w(v,k),\quad
F_w(v,k)=H_{k+2d_{v,w}}^{G(w)\times T}\big(\frakM(v,w)\big).$$
Let
$|w\rangle$ be the fundamental class of $\frakM(0,w)$ in $F_w.$
Let $A_w$ be the $\bbZ^I\times \bbZ$-graded $\Bbbk[w]$-algebra 
given by 
\begin{align*}
A_w&=\bigoplus_{v_2\in\bbZ^I} A_{w}(v_2)
=\bigoplus_{v_2\in\bbZ^I}\bigoplus_{k_2\in\bbZ} A_{w}(v_2,k_2)\subset\End_{\Bbbk[w]}(F_w),
\end{align*}
where $A_{w}(v_2,k_2)$ consists of all $\Bbbk[w]$-linear endomorphisms of $F_w$ which are homogeneous of degree 
$(v_2,k_2)$. By \cite{N01}, the $\Bbbk[w]$-module $F_w(v)$ is free of finite rank. 
Write $$F_w^\vee=\Hom_{\Bbbk[w]}(F_w\,,\,\Bbbk[w])
=\prod_vF_w(v)^\vee.$$
Then, we have
\begin{align*}
A_w=\prod_v\bigoplus_{u}F_w(u)\otimes_{\Bbbk[w]}F_w(v)^\vee
\end{align*}
Let $A_w^f\subset A_w$ be the set of finite rank endomorphisms of $F_w$, which is equal to
\begin{align*}
A_w^f=F_{w}\otimes_{\Bbbk[w]} F_{w}^\vee.
\end{align*}
Under the convolution, any cycle in $H_*^{G(w)\times T}\big(\frakM(u,w)\times\frakM(v,w)\big)$
which is proper over $\frakM(u,w)$ gives rise to a $\Bbbk[w]$-linear operator
$F_w(v)\to F_w(u),$ see \cite{CG} for details.

\medskip

\section{Definition of $\bbY$ and $\bfY^1$}\label{sec:def}

The aim of this section is to review the constructions of Maulik and Okounkov \cite{MO} and \cite{SV17a}.
The sections \ref{sec:stable}, \ref{sec:res} and \ref{sec:rmatrix} are a reminder from
 \cite{MO} to which we refer for more details. Section~\ref{sec:Y1} is a reminder of \cite{SV17a}.

\smallskip

\subsection{The stable envelope}\label{sec:stable}

\subsubsection{} 
Fix dimension vectors $v$, $w$, $w_1$, $w_2$ such that $w=w_1+w_2$.
We'll abbreviate 
\begin{align}\label{50}
\G_\sp=G(w_1)\times G(w_2)\times T_\sp,\quad
\G=\G_\sp\times T_\dil,\quad A=\gamma(\bbG_m)
\end{align}
where $\gamma$ is the cocharacter of $G(w_1)\times G(w_2)$ given by
\begin{align}\label{gamma2}
\gamma(t)=1_{w_1}\oplus t^{-1}1_{w_2}.
\end{align}
The group $\G$ acts on the varieties $\frakM(v,w)$ and $\frakM_0(v,w)$.
Let $\diamond$ be the $\bbG_m$-action on $\frakM(v,w)$ associated with the cocharacter $\gamma$.
Let $\frakM(v,w)^\diamond$ be the fixed points locus and
$\frakL(v,w)^\diamond$ be the \emph{attracting set}
\begin{align}\label{atL}\frakL(v,w)^\diamond=\{x\in \frakM(v,w)\,;\,\exists\lim_{t\to 0}t\diamond x\}.
\end{align}
Let $\frakM_0(v,w)^\diamond$ and $\frakL_0(v,w)^\diamond$ be the fixed points locus and the attracting set in $\frakM_0(v,w)$.
As $\pi$ is proper, we have $\frakL(v,w)^\diamond=\pi^{-1}(\frakL_0(v,w)^\diamond).$
Hence, there is a composed homomorphism
$$\frakL(v,w)^\diamond\to\frakL_0(v,w)^\diamond\to\frakM_0(v,w)^\diamond.$$
Consider the (reduced) fiber product given by
$$\frakZ(v,w)=\frakL(v,w)^\diamond\times_{\frakM_0(v,w)^\diamond}\frakM(v,w)^\diamond.$$
Since the $\diamond$-action preserves the symplectic form of $\frakM(v,w)$, 
the fixed point locus  $\frakM(v,w)^\diamond$ admits also a natural symplectic form.
Then $\frakZ(v,w)$ is a Lagrangian closed subvariety of the symplectic manifold 
$\frakM(v,w)^{\op}\times\frakM(v,w)^\diamond$.
The \emph{stable envelope} is the $H^*_\G$-linear map
$$H^\G_*(\frakM(v,w)^\diamond)\to H^{\G}_*(\frakM(v,w))$$
given by the convolution with a $\G$-equivariant Lagrangian cycle 
$$\stab\in H_*^\G(\frakM(v,w)^\op\times\frakM(v,w)^\diamond)$$
which is supported on $\frakZ(v,w)$ and proper over $\frakM(v,w)$. 

\smallskip

\subsubsection{} 
Let $\Irr(\frakM(v,w)^\diamond)$ be the set of connected components of $\frakM(v,w)^\diamond$.
For each $\calX\in\Irr(\frakM(v,w)^\diamond)$, the restriction of the $\G_\sp$-equivariant Euler class of the normal 
bundle $N_{\calX/\frakM}$
to any point of $\calX$ is
$(-1)^{\codim_\frakM\calX/2}$ times a square in $H^*_{\G_\sp}$.
A choice of a square root 
$$\varepsilon_\calX\in H^{\codim_\frakM\calX}_{\G_\sp}$$
of this square for each component $\calX$ is called a \emph{polarization}.
Since $\frakM(v,w)$ is a quiver variety, it is the Hamiltonian reduction of the cotangent bundle to an $\G_\sp$-manifold.
Therefore, it admits a canonical polarization, and any other polarization differs from it by a collection of signs.
In other words, we'll view $\varepsilon_\calX$ either as a class in $H^*_{\G_\sp}$ 
or as a sign according to the context.

\smallskip

\subsubsection{} 
Define the varieties
$\frakM(v,w)^\diamond_{\bbA^1},$ $\frakL(v,w)^\diamond_{\bbA^1}$ and $\frakZ(v,w)_{\bbA^1}$
as above, by replacing 
$\frakM(v,w)$, $\frakM_0(v,w)$ by $\frakM(v,w)_{\bbA^1}$, $\frakM_0(v,w)_{\bbA^1}$ at each step.
There is a (non continuous) map 
$$\nu:\frakL(v,w)^\diamond_{\bbA^1}\to\frakM(v,w)^\diamond_{\bbA^1}$$ 
given by the limit $t\to 0$ in \eqref{atL}.
Since
$\frakM(v,w)_{\bbG_m}$ and $\frakM_0(v,w)_{\bbG_m}$ are non empty isomorphic affine varieties,
we have an isomorphism $\frakM_0(v,w)_{\bbG_m}^\diamond\simeq\frakM(v,w)^\diamond_{\bbG_m}$.
Hence, for each component $\calX\in\Irr(\frakM(v,w)^\diamond_{\bbG_m})$ the set
$\nu^{-1}(\calX)$ is closed in $\frakM(v,w)_{\bbG_m}$ and $\frakZ(v,w)_{\bbG_m}$ is isomorphic to
the closed subvariety of $\frakM(v,w)_{\bbG_m}$ given by
$$\nu^{-1}(\frakM(v,w)^\diamond_{\bbG_m})=\bigsqcup_\calX \nu^{-1}(\calX).$$
We consider the $\G$-invariant cycle 
in  $\frakZ(v,w)_{\bbG_m}$ given by
$$\stab_{\bbG_m}=\sum_\calX\varepsilon_\calX\,[\,\nu^{-1}(\calX)\times_\calX\calX\,]. $$
The inclusion $\{0\}\subset\bbA^1$ gives rise to a specialization map to
$H_*^\G(\frakM(v,w)^\op\times\frakM(v,w)^\diamond)$ as in \cite[\S 2.6.30]{CG}, 
\cite[\S 11.1]{F98}. Let us denote it by the symbol ${\displaystyle\lim_0}$. It yields a Lagrangian cycle
\begin{align}\label{stab0}
\stab=\lim_{0}\,\stab_{\bbG_m}.
\end{align}

\smallskip

\subsubsection{} 
Given dimension vectors $v_1$, $v_2$ such that $v=v_1+v_2$, 
let $\frakM[v_1,v_2]$ and $\frakL[v_1,v_2]$ be the subsets of 
$\frakM(v,w)^\diamond$ and $\frakL(v,w)^\diamond$ attached to
the cocharacter $\rho$ of $G(v)$ given by 
\begin{align}\label{rho2}\rho(t)=1_{v_1}\oplus t^{-1}1_{v_2}.\end{align}
We have 
\begin{align}\label{prod}\begin{split}
\frakM(v,w)^\diamond&=\bigsqcup_{v_1+v_2=v}\frakM[v_1,v_2],\quad 
\frakM[v_1,v_2]=\frakM(v_1,w_1)\times\frakM(v_2,w_2),\\
\frakL(v,w)^\diamond&=\bigsqcup_{v_1+v_2=v}\frakL[v_1,v_2],\quad\frakL[v_1,v_2]=\nu^{-1}(\frakM[v_1,v_2]),\\
\frakZ(v,w)&=\bigsqcup_{v_1+v_2=v}\frakZ[v_1,v_2],\quad 
\frakZ[v_1,v_2]=\frakZ(v,w)\cap \big(\,\frakM(v,w)\times \frakM[v_1,v_2]\,\big).
\end{split}
\end{align}
Note that $\frakM_0(v,w)^\diamond$ is not the disjoint union of the varieties
$\frakM_0(v_1,w_1)\times\frakM_0(v_2,w_2).$
Let $|v|$ denote the sum of the entries of the dimension vector $v$. We have
\begin{align*}\frakM[v_1,v_2]\cap\overline{\frakL[u_1,u_2]}\neq\emptyset\iff 
\big(|v_2|\leqslant |u_2|\ \text{and}\ v_1+v_2=u_1+u_2\big).\end{align*}
If this holds we may also write
\begin{align}\label{order}\frakM[v_1,v_2]\leqslant \frakM[u_1,u_2].\end{align}
Given  $u_1$, $u_2$ with $v=u_1+u_2$  we abbreviate 
\begin{align*}
\frakM[u_1,u_2\,;\,v_1,v_2]&=\frakM[u_1,u_2]\times\frakM[v_1,v_2].
\end{align*}
The varieties above admit deformed version 
by replacing $\frakM(v,w)$ by $\frakM(v,w)_{\bbA^1}$ at each step.
We'll denote them by $\frakZ[v_1,v_2]_{\bbA^1}$, $\frakL[v_1,v_2]_{\bbA^1}$, etc.
Consider the equivariant Lagrangian cycle
\begin{align}\label{stab00}
\stab[v_1,v_2]=\varepsilon_{v_1,v_2}\,\lim_{0}\,[\,\frakZ[v_1,v_2]_{\bbG_m}\,]\in 
H_*^\G\big(\frakM(v,w)\times\frakM[v_1,v_2]\big),
\end{align}
where $\varepsilon_{v_1,v_2}=\varepsilon_{\frakM[v_1,v_2]}$.
To avoid any confusion we may write 
\begin{align*}
\frakZ[v_1,v_2\,;\,w_1,w_2]&=\frakZ[v_1,v_2],\\
\frakL[v_1,v_2\,;\,w_1,w_2]&=\frakL[v_1,v_2],\\
\stab[v_1,v_2\,;\,w_1,w_2]&=\stab [v_1,v_2].
\end{align*}

\smallskip

\subsection{Residues}\label{sec:res}

\subsubsection{} 
Since the group $A$ acts trivially on the fixed points locus 
$\frakM(v,w)^\diamond$, there is a canonical increasing filtration
on the space $H_k^\G(\frakM(v,w)^{\diamond}\times\frakM(v,w)^\diamond)$ whose $d$-th term is
$$\bigoplus_{\ell\leqslant d}H_A^{\leqslant\ell}\cap H_{k+\ell}^{G/A}(\frakM(v,w)^{\diamond}\times\frakM(v,w)^\diamond).$$
The associated graded is a graded $H^*_A$-module whose degree $d$ term is equal to
$$H_A^{d}\otimes H_{k+d}^{G/A}(\frakM(v,w)^{\diamond}\times\frakM(v,w)^\diamond).$$
There is a canonical $\bbZ$-graded ring isomorphism
$$\Bbbk[\varpi]=H^*_{A\times T},\quad \deg(\varpi)=2$$
where $\varpi$ is identified with
the first Chern class of the linear character of $A$.
The inclusion 
$$i:\frakM(v,w)^\diamond\times\frakM(v,w)^\diamond\subset\frakM(v,w)\times\frakM(v,w)^\diamond$$
is an l.c.i.~morphism. Hence the pullback $i^*$ in equivariant Borel-Moore homology is well defined.
Given connected components $\calX,$ $ \calY$ of $\frakM(v,w)^\diamond$ such that $\calX<\calY$, we denote by 
$\bullet|_{\calX\times\calY}$ the restriction to $\calX\times\calY$.
By \cite[prop.~4.8.2]{MO}, 
there is a unique $\G/A$-equivariant Lagrangian cycle 
$\bfr_{\calX\times\calY}$ on $\calX^{\op}\times\calY$ 
such that
\begin{align}\label{res}
\stab|_{\calX\times\calY}=\hbar \varpi^{\codim_\frakM\calX/2-1}\cap\bfr_{\calX\times\calY}\ \text{mod}\ 
H_A^{<\,\codim_\frakM\calX-2}\cap H_*^{\G/A}(\calX\times\calY).
\end{align}
We'll call it the \emph{residue} of the cycle $\stab$.
It is supported on the closed subset
$$(\calX\times\calY)\cap(\frakM(v,w)^{\diamond}\times_{\frakM_0(v,w)}\frakM(v,w)^\diamond).$$
In particular, it is proper over $\calX$. 
We'll also write $\bfr_{\calX\times\calY}$ for its image in $H_*^{\G_\sp/A}(\calX\times\calY).$

\smallskip

\subsubsection{} 
By \cite[lem.~3.4.2]{MO}, for any $\G_\sp$-invariant Lagrangian cycle 
$C$ in $\frakM(v,w)^\op\times\frakM(v,w)^\diamond$ 
there is a unique $\G_\sp/A$-equivariant Lagrangian cycle 
$\res(C)$ in $\frakM(v,w)^{\diamond,\op}\times\frakM(v,w)^\diamond,$ 
called the \emph{Lagrangian residue} of $C$, such that
\begin{itemize}
\item
$\res(C)$ is supported on $C\cap(\frakM(v,w)^\diamond\times\frakM(v,w)^\diamond)$,
\item 
$C|_{\calX\times\calY}=\varepsilon_{\calX}\cap\res(C)|_{\calX\times\calY}\ \text{modulo}\ 
H_A^{<\,\codim_\frakM\calX/2}\otimes H_*^{\G_\sp/A}(\calX\times\calY),
$
\end{itemize}
where $\calX, \calY\in\Irr(\frakM(v,w)^\diamond)$
and $\varepsilon_{\calX}\in H^*_{\G_\sp}$ is the polarization.
The residue above is not the same as the Lagrangian residue of stab.
Indeed, the characterization of the stable envelope in \cite[thm.~3.3.4]{MO} implies that the off-diagonal terms of the
Lagrangian residue of $\stab$ are zero.

\smallskip

\subsection{Hecke correspondences and residues}\hfill\\
\label{sec:HSE1}

Let $v_1$, $v_2$, $v$, $w_1$, $w_2$, $\gamma$, $\rho$ be as above.
Fix an $I$-graded subspace $V_1\subseteq V$ with dimension vector $v_1$. 
Let $P$ be the corresponding standard parabolic subgroup of $G(v)$ and
$L=G(v_1)\times G(v_2)$ be the standard Levi subgroup of $P$.
Fix $I$-graded linear isomorphisms $\k^{v_1} \simeq V_1$ and $\k^{v_2} \simeq V/ V_1$.

\subsubsection{} 
The geometric quotient
\begin{align}\label{H}
\begin{split}
\frakh[v_1,v_2\,;\,w]&=\H[v_1,v_2\,;\,w]/\!\!/P,\\
\H[v_1,v_2\,;\,w]&=\{z\in\M_s(v,w)\,;\,z(V_1\oplus W)\subseteq V_1\oplus W\}
\end{split}
\end{align}
is called a \emph{Hecke correspondence}. 
Given a locally closed subset 
$\SS\subseteq \M(v_2)$ we set
\begin{align*} 
\begin{split}
\frakh[v_1,\SS\,;\,w]&=\H[v_1,\SS\,;\,w]/\!\!/P,\\
\H[v_1,\SS\,;\,w]&=\{(\bar x,\bar a)\in\H[v_1,v_2\,;\,w]\,;\,\bar x_2\in \SS\},
\end{split}
\end{align*}
where $\bar x_2$ is the representation of $\Pi$ on $V/V_1$ equal to $z$ modulo $V_1\oplus W$.
If no confusion is possible we abbreviate
\begin{align*}
\H[v_1,\SS]&=\H[v_1,\SS\,;\,w],\quad
\frakh[v_1,\SS]=\frakh[v_1,\SS\,;\,w].
\end{align*}
By \cite[lem.~3.14, prop.~�3.15]{SV17a}, the Hecke correspondence
$\frakh[v_1,\SS\,;\,w]$ is an isotropic subvariety of $\frakM(v,w)^{\op}\times\frakM(v_1,w)$
if $\SS$ is any isotropic subvariety of the symplectic vector space $\R(v_2)$ and
it
is a closed Lagrangian
local complete intersection if $\SS=\Lambda_{(v_2)}.$

\smallskip

\subsubsection{}
By \cite[prop.~3.13]{SV17a}, the assignment $z\mapsto(z,z|_{V_1\oplus W})$ is a closed embedding
\begin{align}\label{hecke}\frakh[v_1,v_2\,;\,w]\subseteq\frakM(v,w)\times\frakM(v_1,w)\end{align}
which is projective over $\frakM(v,w)$.
Consider the composed map
\begin{align}\label{map}\frakh[v_1,v_2\,;\,w_1]\to\frakM_0(v_2,0)\subset\frakM_0(v_2,w_2).\end{align} 
We define the  \emph{generalized Hecke correspondence} to be the (reduced) fiber product
$$\frakeh[v_1,v_2\,;\,w_1,w_2]=\frakh[v_1,v_2\,;\,w_1]\times_{\frakM_0(v_2,w_2)}\frakM(v_2,w_2).$$
By \eqref{hecke}, there is a closed embedding
$$\frakeh[v_1,v_2\,;\,w_1,w_2]\subseteq \frakM[v,0]\times_{\frakM_0(v,w)}\frakM[v_1,v_2].$$ 
There is also an isomorphism
\begin{align}\label{GH}
\frakeh[v_1,v_2\,;\,w_1,w_2]&=\GH[v_1,v_2]/\!\!/P\times L,
\end{align}
where $\GH[v_1,v_2]$ is the set of triples 
\begin{align}\label{triple}
(z,z'_1,z'_2)\in 
\H[v_1,v_2\,;\,w_1]\times\M_s(v_1,w_1)\times \M_s(v_2,w_2)
\end{align} 
satisfying the following conditions
\begin{itemize}
\item $z|_{V_1\oplus W_1}\simeq z'_1$ as $\tilde\Pi$-modules, 
\item $a'_2=0$,
\item 
$\pi(\bar x_2)\simeq\pi(\bar x'_2)$ as $\Pi$-modules.
\end{itemize}
Here, we write
$z=(\bar x,\bar a)$, $z'_2=(\bar x'_2,\bar a'_2)$ and
$\bar x_2=\bar x\ \text{mod}\  V_1$ in $\M(v_2)$.
Further, the symbol $\pi$ denotes the semisimplification of $\Pi$-modules.

\smallskip

More generally, for each locally closed subset $\SS\subset\M(v_2)$ and for each variety $\T$ over $\frakM_0(v_2,w_2)$,
we consider the (reduced) fiber product 
$$\frakeh[v_1,\SS\,;\,w_1,\T]=\frakh[v_1,\SS\,;\,w_1]\times_{\frakM_0(v_2,\delta_i)}\T.$$
We may abbreviate
\begin{align*}
\frakeh[v_1,v_2\,;\,w_1,\T]&=\frakeh[v_1,\M(v_2)\,;\,w_1,\T],\\
\frakeh[v_1,\SS\,;\,w_1,w_2]&=\frakeh[v_1,\SS\,;\,w_1,\frakM(v_2,w_2)],\\
\frakeh[v_1,v_2]&=\frakeh[v_1,v_2\,;\,w_1,w_2],\\
&=\frakeh[v_1,\M(v_2)\,;\,w_1,w_2].
\end{align*}

\smallskip

\subsubsection{} 
We'll also need the following geometric quotient
\begin{align}\label{P}
\begin{split}
\frakp[v_1,v_2]&=\P[v_1,v_2]/\!\!/P,\\
\P[v_1,v_2]&=\{z\in\M_s(v,w)\,;
\,z(V_1\oplus W_1)\subseteq V_1\oplus W_1\}.
\end{split}
\end{align}
From the proof of \cite[prop.~3.13]{SV17a}, we deduce that the assignment $z\mapsto(z,z|_{V_1\oplus W_1})$ is a 
closed embedding
\begin{align}\label{emb}
\frakp[v_1,v_2]\subseteq\frakM(v,w)\times\frakM(v_1,w_1).
\end{align}
There is also an obvious map
\begin{align*}\frakp[v_1,v_2]\to\frakM_0(v_2,w_2)
\end{align*}
which yields the (reduced) fiber product
\begin{align*}
\frakep[v_1,v_2]
&=\frakp[v_1,v_2]\times_{\frakM_0(v_2,w_2)}\frakM(v_2,w_2).
\end{align*}
Using \eqref{emb} we get a closed embedding
\begin{align}\label{emb-gp}\frakep[v_1,v_2]
\subseteq\frakM(v,w)\times\frakM[v_1,v_2].
\end{align}
We also have an isomorphism
\begin{align}\label{GP}
\frakep[v_1,v_2]&=\GP[v_1,v_2]/\!\!/P\times L,
\end{align}
where $\GP[v_1,v_2]$ is the set of triples
$$(z,z'_1,z'_2)\in 
\P[v_1,v_2]\times\M_s(v_1,w_1)\times \M_s(v_2,w_2)$$ 
satisfying the following conditions
\begin{itemize}
\item $z|_{V_1\oplus W_1}\simeq z'_1$ as a $\tilde\Pi$-module, 
\item $\pi(z\ \text{mod}\  V_1\oplus W_1)\simeq\pi(z'_2)$ as $\Pi$-modules on $V\oplus W/V_1\oplus W_1$.
\end{itemize}
From \eqref{H}, \eqref{GH}, \eqref{P} and \eqref{GP} we deduce that
\begin{align}\label{hp}
\begin{split}
\frakh[v_1,v_2\,;\,w_1]&=\frakp[v_1,v_2]\cap\big(\frakM[v,0]\times\frakM(v_1,w_1)\big),\\
\frakeh[v_1,v_2]&=\frakep[v_1,v_2]\,\cap\,\frakM[v,0\,;\,v_1,v_2].
\end{split}
\end{align}

\smallskip

\subsubsection{} 
We can now consider the residue of the stable envelope.
Given dimension vectors $u_1$, $u_2$ with $v=u_1+u_2$  and $|u_2|< |v_2|$, we abbreviate
\begin{align}\label{stab00}
\begin{split}
\bfr[u_1,u_2\,;\,v_1,v_2]&=\bfr_{\,\frakM[u_1,u_2\,;\,v_1,v_2]}.
\end{split}
\end{align}
To avoid any confusion we may also write 
\begin{align*}
\bfr[u_1,u_2\,;\,v_1,v_2\,;\,w_1,w_2]=\bfr [u_1,u_2\,;\,v_1,v_2].
\end{align*}
For a future use, let us focus on the case $u_1=v$ and $u_2=0$. Then, we have
$$\bfr[v,0\,;\,v_1,v_2]\in H_{\mid}^{\G_\sp/A}\big(\frakM[v,0\,;\,v_1,v_2]\big),$$
where the degree $\mid$ is given by
\begin{align}\label{e}\mid=d_{v,w_1}+d_{v_1,w_1}+d_{v_2,w_2}.\end{align}
It is an equivariant Lagrangian cycle of the symplectic manifold
$$\frakM[v,0\,;\,v_1,v_2]=\frakM(v,w_1)^\op\times\frakM(v_1,w_1)\times\frakM(v_2,w_2).$$

\smallskip

\begin{proposition}\label{prop:D1}\hfill
\begin{itemize}[leftmargin=8mm]
\item[$\mathrm{(a)}$] 
$\bfr[v,0\,;\,v_1,v_2]$  is supported on the closed subset $\frakeh[v_1,v_2]$ of $\frakM[v,0\,;\,v_1,v_2]$.
\item[$\mathrm{(b)}$] 
$\stab[v_1,v_2]$  is supported on the closed subset 
$\frakep[v_1,v_2]$ of $\frakM(v,w)\times\frakM[v_1,v_2]$.
\end{itemize}
\end{proposition}

\smallskip

\begin{proof}
Replacing everywhere quiver varieties by deformed ones we define
varieties $\frakp[v_1,v_2]_{\bbA^1}$ and
$\frakep[v_1,v_2]_{\bbA^1}$  such that
$\frakp[v_1,v_2]$ and $\frakep[v_1,v_2]$ are the fibers at 0 of
$\frakp[v_1,v_2]_{\bbA^1}$ and $\frakep[v_1,v_2]_{\bbA^1}.$
Then, we have a closed embedding
\begin{align}\label{emb-gp-D}\frakep[v_1,v_2]_{\bbA^1}
\subseteq\frakM(v,w)_{\bbA^1}\times\frakM[v_1,v_2]_{\bbA^1}.
\end{align}
Now, for each representation $z\in\R(v,w)$ preserving the subspace
$V_1\oplus W_1$ of $V\oplus W$, let
$z_1$ be the restriction of $z$ to $V_1\oplus W_1$  and
$z_2$ be the induced representation on the quotient space
$(V\oplus W)/(V_1\oplus W_1)$.
Then, by \cite[\S 3.7.4]{SV17a} we have
\begin{align}\label{L}
\begin{split}
\frakL[v_1,v_2]_{\bbA^1}&=\L[v_1,v_2]_{\bbA^1}/\!\!/P,\\
\L[v_1,v_2]_{\bbA^1}&=\{z\in\P[v_1,v_2]_{\bbA^1}\,;\,
z_2\in \M_s(v_2,w_2)\}.
\end{split}
\end{align}
Therefore, we have an obvious inclusion
\begin{align}\label{inclusion}\frakL[v_1,v_2]_{\bbA^1}\subseteq\frakp[v_1,v_2]_{\bbA^1},
\end{align}
hence a chain of inclusions
\begin{align}\label{LHS}
\frakL[v_1,v_2]_{\bbG_m}\times_{\frakM^\diamond_{0,\bbG_m}}\frakM[v_1,v_2]_{\bbG_m}
\subseteq\frakep[v_1,v_2]_{\bbG_m}
\subseteq\frakep[v_1,v_2]_{\bbA^1}.
\end{align}
By \eqref{stab0} the cycle $\stab[v_1,v_2]$ is supported on the closure of the left hand side of \eqref{LHS} in
$$\frakM(v,w)_{\bbA^1}\times\frakM[v_1,v_2]_{\bbA^1}.
$$
Since \eqref{emb-gp-D} is a closed embedding,
the cycle $\stab[v_1,v_2]$ is supported on the set $\frakep[v_1,v_2],$
proving part (b).
From part (b) we deduce that the cycle $\bfr[v,0\,;\,v_1,v_2]$ is supported on 
$$\frakep[v_1,v_2]\,\cap\,\frakM[v,0\,;\,v_1,v_2].$$
Thus, part (a) follows from \eqref{hp}.

\end{proof}

\smallskip

\subsection{Definition of $\bbY$}
\label{sec:rmatrix}

\subsubsection{} 
Let $v$, $w_1$, $w_2$, $\gamma$ be as above.
Recall that
\begin{align*}
A_{w_1}\otimes A_{w_2}
&\subset\End\big(\bigoplus_{v_1,v_2}H_*^{G(w_1)\times T}(\frakM(v_1,w_1)\otimes_\Bbbk
H_*^{G(w_2)\times T}(\frakM(v_2,w_2)\big)\\
&\subset\End\big(\bigoplus_{v_1,v_2}H_*^{\G}(\frakM[v_1,v_2])\big).
\end{align*}
Maulik and Okounkov have defined 
an \emph{$R$-matrix} 
which is a formal series in $(A_{w_1}\otimes A_{w_2})[[\varpi^{-1}]]$ of the form
$$R_{w_1,w_2}(\varpi)=1+\hbar\sum_{l>0}R_{w_1,w_2\,;\,l}\,\varpi^{-l}.$$
The $R$-matrix is homogenous of degree zero relatively to the $\bbZ^I\times\bbZ$-grading. 
The \emph{classical $R$-matrix} 
is the first Fourier coefficient
$\bfr_{w_1,w_2}=R_{w_1,w_2\,;\,1}.$
We have the decomposition
\begin{align}\label{decomp}
\bfr_{w_1,w_2}=\bfr_{w_1,w_2,-}+\bfr_{w_1,w_2,0}+\bfr_{w_1,w_2,+}
\end{align}
with
$$\bfr_{w_1,w_2,-}=\prod_\calY\sum_{\calX<\calY}\bfr_{\calX,\calY},\quad
\bfr_{w_1,w_2,0}=\prod_\calY\bfr_{\calY,\calY},\quad
\bfr_{w_1,w_2,+}=\prod_\calY\sum_{\calX>\calY}\bfr_{\calX,\calY},$$
where $\calX$, $\calY$ are connected components of $\frakM(v,w)^\diamond$ and
$\bfr_{\calX,\calY}$ is a class in $H_*^{\G_\sp/A}(\calX\times\calY)$ which can be viewed as an operator
$H_*^{\G_\sp/A}(\calY)\to H_*^{\G_\sp/A}(\calX)$ 
via the convolution product.
The partial order on the set $\Irr(\frakM(v,w)^\diamond)$ is as in \eqref{order}.
The first term in \eqref{decomp} is given by
$$\bfr_{\frakM[u_1,u_2]\,,\,\frakM[v_1,v_2]}=\bfr[u_1,u_2\,;\,v_1,v_2\,;\,w_1,w_2]$$
for each dimension vectors $u_1,u_2,v_1,v_2$ such that $u_1=u_2=v_1=v_2=v$ and
$\frakM[u_1,u_2]<\frakM[v_1,v_2].$
It is viewed as a linear operator in 
$A_{w_1}(u_1-v_1)\otimes A_{w_2}(u_2-v_2)$.

\smallskip

\subsubsection{} 
Now, we choose
\begin{align}\label{w}w_1=w,\quad w_2=\delta_i,\quad 
\G=G(w)\times G(\delta_i)\times T=G(w)\times A\times T,\end{align}
where $A$ is as in \eqref{50}.
Then, we define
\begin{align*}
F_i=H_*^T(\frakM(\delta_i)),\quad F_i^\vee=\Hom(F_i,\Bbbk),\quad
A_i^f=F_i\otimes F_i^\vee.
\end{align*}
Taking the frace over $F_i$ yields a $\Bbbk$-linear map $A_i^f\to\Bbbk$.

\smallskip

Taking the formal variable $u$ to the first equivariant Chern class of the linear character of 
$G(\delta_i)$ yields isomorphisms $\Bbbk[u]\simeq\Bbbk[\delta_i]$, $F_i[u]\simeq F_{\delta_i}$,
$A_i^f[u]\simeq A^f_{\delta_i}$ and
$$\Bbbk[w][u]\to H_*^\G=\Bbbk[w][\varpi].$$
Hence, if $l>0$, $v\in\bbZ^I$ and $m\in A_{i}^f(v)$, the trace over $F_i$ yields an element
$$\tr_{F_{i}}\big((1\otimes m)\,R_{w,\delta_i\,;\,l}\big)\in A_w(v).$$
So, for each polynomial $m(u)$ in $A^f_{i}[u]$,
there is a unique element $\E(m(u))$ in $\prod_{w}A_w$
which acts on $F_w$ as  the operator
$$\E(m(u),w)=\Res_{u=\infty}\tr_{F_{i}}\Big((1\otimes m(u))\,R_{w,\delta_i}(u)\Big)\Big/\hbar.$$

\smallskip

\begin{definition}
Let $\bbY_Q$ be the $\bbZ^I\times\bbZ$-graded $\Bbbk$-subalgebra of $\prod_{w}A_w$
generated by the elements $\E(m(u))$ as $m(u)$ runs over $A^f_{i}[u]$.
\end{definition}

\smallskip

\noindent There is a surjective $\bbZ^I\times\bbZ$-graded $\Bbbk$-algebra homomorphism
\begin{align*}
\E\,:\,\text{Tensor\ Algebra}\Big(\bigoplus_i A_{i}^f[u]\Big)\Big/\!\sim\,\to\bbY_Q,\quad
m(u)\mapsto\E(m(u)),
\end{align*}
where $\sim$ is the two-sided ideal generated by the RTT=TTR relations \cite[(5.10)]{MO}.

\smallskip

\subsubsection{} 
Since $\G_\sp/A=G(w)\times T_\sp$, we have
$$\bfr[u_1,u_2\,;\,v_1,v_2]\in H_*^{G(w)\times T_\sp}(\frakM[u_1,u_2\,;\,v_1,v_2]).$$
Any element $m\in A^f_{i}$ can be viewed as a constant polynomial in $u$, yielding an operator $\e(m)=\E(m)$
which
acts on $F_w$ as the operator
$$\e(m\,;\,w)=\tr_{F_{i}}\big((1\otimes m)\,\bfr_{w,\delta_i}\big)\in A_w.$$
Let $\mo_Q$ be the $\Bbbk$-submodule of $\bbY_Q$ spanned by the elements $\e(m)$ above.
The commutator is a Lie bracket on $\bbY_Q$ for which $\mo_Q$ is a Lie subalgebra.
The Yang-Baxter equation implies that the map $m\mapsto\e(m)$ is a surjective $\bbZ^I\times\bbZ$-graded Lie 
algebra homomorphism
$$\e\,:\,\bigoplus_i A_{i}^f\to \mo_Q,$$
such that the Lie bracket on the left hand side is given by
$$[m_j\,,\,m_i]=(\tr_{F_{j}}\otimes 1)([\bfr_{\delta_j,\delta_i}\,,\,m_j\otimes m_i]),
\quad\forall m_i\in A_{i}^f,\quad\forall m_j\in A_{j}^f.
$$
Let $\calV_i$ and $\calW_i$ be the \emph{universal} and the \emph{tautological} 
$G(w)\times T$-equivariant vector bundles on $\frakM(v,w)$, which are given by
\begin{align}\label{bundle}\begin{split}
\calV_i=\M_s(v,w)\times_{G(v)} V_i,\quad \calW_i=\frakM(v,w)\times W_i.
\end{split}\end{align}
Let 
$\psi_{i,l}$ and $\phi_{i\,,\,l}$ be the operators in $\prod_wA_w$ given by the cap product 
with the $T$-equivariant cohomology class $\ch_l(\calV_i)$ and $\ch_l(\calW_i)$.
The following is proved in \cite{MO}~:

\smallskip

\begin{itemize}
\item  
$\mo_Q(v)$ is a free $\bbZ$-graded $\Bbbk$-submodule of finite rank, 

\smallskip



\item 
$\mo_Q=\mo_-\oplus\mo_0\oplus\mo_+$ where
$$\mo_-=\bigoplus_{v<0}\mo_Q(v),\quad \mo_0=\mo_Q(0),\quad\mo_+=\bigoplus_{v>0}\mo_Q(v)$$
 and $v>0$ if and only if $v\in\bbN^I$ and $v\neq 0.$

\smallskip

\item 
there are unique elements $h_{i,l}$ and $z_{i,l}$
in $\bbY_Q$ which act on $F_w$ via the 
operator of cap product with the classes $\psi_{i,l}$ and $\phi_{i,l}$.
The element $z_{i,l}$ is central.
We have a triangular decomposition 
$\bbY_Q\simeq\bbY_-\otimes\bbY_0\otimes\bbY_+$
such that the $\Bbbk$-algebra $\bbY_\pm$
is generated by $\mo_\pm$ and the adjoint action of $\{h_{i,l}\}$.
We also set $\bbY_{\geqslant }=\bbY_0 \otimes \bbY_+$.

\end{itemize}

\smallskip

\noindent
Unless there is a risk of confusion we will simply write $\bbY_Q=\bbY$ and $\mo_Q=\mo$.
We'll also write $\bbY_{\geqslant, K}=\bbY_\geqslant\otimes K$, $\bbY_{+,K}=\bbY_+\otimes K$, etc.

\medskip

\subsection{Definition of $\bfY^1$}\label{sec:Y1}

\subsubsection{} 
Let $\Y^1$ be the $\bbZ^I\times \bbZ$-graded $\Bbbk$-module given by
$$\Y^1=\bigoplus_{v}\Y^1(v),\quad
\Y^1(v)=\bigoplus_{k}\Y^1(v,k),\quad
\Y^1(v,k)=H_{k+2 d_v}^{G(v)\times T}(\Lambda^1(v)).$$
For each dimension vector $v$ the $\Bbbk$-module $\Y^1(v)$  
has a canonical $\Bbbk[v]$-module stucture given by the cap-product $\cap$.
We have defined a $\Bbbk$-algebra structure on $\Y^1$ in \cite[\S 5.1]{SV17a}.

\smallskip

\subsubsection{} 
Let $\Bbbk(\infty)$ 
be MacDonald's ring of symmetric functions with coefficients in $\Bbbk$.
It is the free commutative $\Bbbk$-algebra generated by the set of power sum polynomials
$\{p_l\,;\,l>0\}$.
It carries a comultiplication $\Delta$, a counit $\eta$, an antipode $S$ and
the $\bbZ$-grading such that the element $p_l$ has the degree $2 l$.
We define 
$\bfY(0)=\Bbbk(\infty)^{\otimes I}.$
For each vertex $i\in I$ and each integer $l>0$, let $p_{i,l}$ be the element of $\bfY(0)$ given by
$$p_{i,l}=(1\otimes\cdots\otimes 1\otimes p_l\otimes 1\otimes \cdots\otimes 1)/l!,$$
where $p_l$ is at the $i$-th spot.
Restricting a representation of $G(\infty^I)$ to the subgroup $G(v)$, yields a 
$\Bbbk$-algebra homomorphism 
$$\bfY(0)\to\Bbbk[v],\quad \rho\mapsto \rho_v.$$
Let $\calU_i$ be the \emph{tautological} 
$G(v)\times T$-equivariant vector bundle on $\Lambda^1(v)$, which is given by
\begin{align*}\begin{split}
\calU_i=\Lambda^1(v)\times V_i,
\end{split}\end{align*}
with the obvious $G(v)\times T$-action. 
There is a $\bfY(0)$-action $\circledast$ on $\Y^1$ such that the element
$p_{i,l}$ acts via the cap product with the class $\ch_l(\calU_i)$.
For each elements $x\in\bfY(0)$ and $y,z\in\Y^1$ we have 
$$x\circledast(y\star z)=\sum(x_1\circledast y)\star(x_2\circledast z).$$
We deduce that $\Y^1$ is a $\bfY(0)$-module $\bbZ^I\times\bbZ$-graded algebra.
The COHA is the smash product $$\bfY^1=\bfY(0)\ltimes\Y^1.$$
It is a free $\bbN^I\times(-\bbN)$-graded $\Bbbk$-algebra such that
$\bfY^1(v,k)=\Y^1(v,k)$ if $v\neq 0$
and $\bfY^1(0,k)$ is the degree $k$ part of $\bfY(0)$. We set $$\bfY_K=\bfY^1 \otimes K.$$

\smallskip

\subsubsection{} 
Let $F_{w}$ be the $\bbZ^I\times \bbZ$-graded $\Bbbk[w]$-module given by
$$F_w=\bigoplus_{v\in\bbN^I}F_w(v),\quad F_w(v)=\bigoplus_{k\in\bbZ}F_w(v,k),\quad
F_w(v,k)=H_{k+2d_{v,w}}^{G(w)\times T}\big(\frakM(v,w)\big).$$
We set
$|w\rangle=[\frakM(0,w)]$.
The following is proved in \cite[prop.~5.19]{SV17a}.

\smallskip

\begin{proposition}\label{prop:A} \hfill
\begin{itemize}[leftmargin=8mm]

\item[$\mathrm{(a)}$] The $\bbZ^I\times \bbZ$-graded $\Bbbk$-algebra $\Y^1$ acts on the 
$\bbZ^I\times \bbZ$-graded $\Bbbk[w]$-module $F_w$.

\item[$\mathrm{(b)}$] The action 
on the element $|v\rangle$ yields an injective map $\Y^1(v)\to F_{v}(v).$ 

\item[$\mathrm{(c)}$]
The representation of $\Y^1$ in $\bigoplus_wF_w$ is faithful.

\item[$\mathrm{(d)}$] The action of
$\Y^1(v)$ on  $|w\rangle$ yields a $\Bbbk[w]$-linear map 
$\Y^1(v)\otimes\Bbbk[w]\to F_{w}(v)$ whose
image is the pushward of $H_*^{G(w)\times T}(\frakL^1 (v,w))$ by the inclusion
$\frakL^1 (v,w)\subset\frakM(v,w)$.
\end{itemize}
\end{proposition}

\smallskip

\subsubsection{} 
Next, given any integer $l>0$, let $\x_{i,l}$ be the element in $\Y^1$ given by
$\x_{i,l}=[\,\Lambda_{(l\,\delta_i)}\,]$ if $q_i>0,$ and
$ \x_{i,l}=\delta_{l,1}\,[\,\Lambda^1(\delta_i)\,]$ if $q_i=0$.
We set
\begin{align*}
I^\re&=\{i\in I\,;\, q_i=0\}, \\
 I^\el&=\{i\in I\,;\, q_i=1\}, \\
  I^\hyp&=\{i\in I\,;\, q_i>1\}.
\end{align*}
Then, the following is proved in \cite[thm.~5.18]{SV17a}.

\begin{proposition}\label{prop:A} \hfill
\begin{itemize}[leftmargin=8mm]

\item[$\mathrm{(a)}$]
The $\Bbbk$-algebra  $\bfY^1$ is generated by the subset
$$\bfY(0)\cup\{\x_{i,1}\,;\,i\in I^\re\}\cup\{\x_{i,l}\,;\,l>0\,,\,i\in I^\el \cup I^\hyp\}.$$
\item[$\mathrm{(b)}$]
The $K$-algebra  $\bfY_K$ is generated by the subset
$$(\bfY(0)\otimes K)\cup\{\x_{i,1}\otimes 1\,;\,i\in I^\re\cup I^\el\}\cup\{\x_{i,l}\otimes 1\,;\,l>0\,,\,i\in I^\hyp\}.$$
\end{itemize}
\end{proposition}

\smallskip

\subsubsection{}
Assume that $v_2=l\,\delta_i$.
Set
\begin{align*}
q_i>0,\quad l>0 &\,\Rightarrow\,
\frakC(v_1,v_2\,;\,w)=\frakh[v_1,\Lambda_{(v_2)}\,;\,w],\\
q_i=0,\quad l=1&\,\Rightarrow \,
\frakC(v_1,v_2\,;\,w)=\frakh[v_1,v_2\,;\,w],\\
q_i=0,\quad l>1&\,\Rightarrow\, \frakC(v_1,v_2\,;\,w)=\emptyset.
\end{align*}
Let $\frakM(w)$ be the disjoint sum of all $\frakM(v,w)$'s and $\frakC_{i,l}(w)$ be the disjoint sum of all
$\frakC(v_1,v_2\,;\,w)$'s where $v,v_1,v_2$ are as above.
We'll view $\frakC_{i,l}(w)$ as a closed $G(w)\times T$-invariant subvariety of  
the symplectic manifold $\frakM(w)^{\op}\times\frakM(w)$ which is proper over $\frakM(w)$. 
Hence, it acts by convolution on $F_w$, yielding a $\Bbbk[w]$-linear operator in $A_w$.
Therefore, for each $i\in I$ and $l\in\bbZ_{>0}$ the family of correspondences $\frakC_{i,l}(w)$ defines an element 
$$\frakC_{i,l}\in\prod_{w}A_{w}.$$
The following is proved in \cite[prop.~�5.20, thm.~5.22]{SV17a}.

\smallskip

\begin{proposition}\label{prop:B} \hfill
\begin{itemize}[leftmargin=8mm]
\item[$\mathrm{(a)}$]
$\x_{i,l}$ acts on $F_w$ via the operator $\frakC_{i,l}(w)$.
\item[$\mathrm{(b)}$]
$\bfY^1$ is isomorphic to the $\Bbbk$-subalgebra of $\prod_wA_w$ generated by
$$\{\psi_{i,l}\,;\,l>0\,,\,i\in I\}\cup\{\frakC_{i,1}\,;\,i\in I^\re\}\cup\{\frakC_{i,l}\,;\,l>0\,,\,i\in I^\el \cup I^\hyp\}.$$
\end{itemize}
\end{proposition}

\smallskip

\begin{remark}\label{rem:3.6}
Assume that $v_2=l\,\delta_i$ and $q_i>1$.
By \cite[prop.~�3.16]{SV17a}, 
 the Hecke correspondence
$\frakh[v_1,\Lambda_{\nu}\,;\,w]$ is either Lagrangian and irreducible or empty for any $\nu \vDash v_2$.
\end{remark}

\medskip

\section{Comparison of $\bfY_K$ and $\bbY_{\geqslant,K}$}

We'll use the same notation and assumptions as in the previous section.

\smallskip

\subsection{The main theorem}\hfill\\

We'll  prove the following.

\smallskip

\begin{theorem}\label{T:main} 
Assume that the torus $T$ contains a one parameter subgroup which scales all the quiver 
data by the same scalar.
Then, there is a $K$-algebra embedding
$\bfY_K\subset\bbY_{\geqslant,K}$ which intertwines the representations of 
$\bbY_{\geqslant,K}$ and $\bfY_K$ in $F_w\otimes K$ given in $ \S\S  \ref{sec:rmatrix}, \ref{sec:Y1}$
for each dimension vector $w$. The inclusion $\bfY_K \subset \bbY_{\geqslant, K}$ 
restricts to an embedding $\Y_K \subset \bbY_{+,K}$.
\end{theorem}

\smallskip

\begin{proof}
By Proposition \ref{prop:B} and the definition of $\bbY$, we can view
$\bfY_K$ and $\bbY_{\geqslant,K}$ as $K$-subalgebras of $\prod_wA_w\otimes K$.
More precisely
$\bfY_K$ is generated by the set
$$\{\psi_{i,l}\,;\,i \in I, l \in \bbN\} \cup\{\frakC_{i,1}\,;\,i\in I^\re\cup I^\el\}\cup\{\frakC_{i,l}\,;\,l>0\,,\,i\in I^\hyp\},$$
and $\bbY_\geqslant\otimes K$ by the set
$$\{\psi_{i,l}\,;\,i \in I, l \in \bbN\} \cup\bigcup_{i\in I}
\e\big(A_{\delta_i}\big).$$
Therefore, the theorem is a consequence of the following proposition
which we will prove in the next two sections.

\smallskip

\begin{proposition} \label{prop:key} 
We have 
$$\big(l=1,\,i\in I^\re\cup I^\el\big)\ \text{or}\ \big(l>0,\,i\in I^\hyp\big) \Rightarrow 
\frakC_{i,l}\in\e\big(A_{i}(l\delta_i)\big).
$$
\end{proposition}

\end{proof}

\smallskip

\begin{remark}\label{conjecture}
We conjecture that the $K$-algebra 
$\bbY_{\geqslant,K}$ is indeed generated by $\bfY_K$ and the central elements $\phi_{i,l}$.
\end{remark}

\smallskip

\subsection{Generalized Hecke correspondences and $R$-matrices}
\label{sec:HSE2}

\subsubsection{}
To prove Proposition \ref{prop:key}, we need more details on the classical $R$-matrix.
Let $w_1$, $w_2$, $w$ are as in \eqref{w}
and
\begin{align*}v_2=l\delta_i,\quad l\in\bbN,\quad i\in I^\hyp.
\end{align*}
Since $q_i>1$,
the generic representation type in $RT(v_2,0)$ is 
$$\kappa_{v_2,0}=(1,v_2).$$
We abbreviate 
\begin{align*}
\SS(v_2)=\M(\kappa_{v_2,0}),\quad
\frakh[v_1,\SS(v_2)]=\frakh[v_1,\SS(v_2)\,;\,w],\quad
\frakeh[v_1,\SS(v_2)]=\frakeh[v_1,\SS(v_2)\,;\,w,\delta_i].
\end{align*}
By Proposition \ref{prop:D1}, the cycle $\bfr[v,0\,;\,v_1,v_2]$  is supported on the generalized Hecke correspondence
$\frakeh[v_1,v_2]$.
The goal of this section is to prove the following technical result.

\smallskip

\begin{proposition}\label{prop:D2}
\hfill
\begin{itemize}[leftmargin=8mm]
\item[$\mathrm{(a)}$] 
$\frakeh[v_1,\SS(v_2)]$ is open in $\frakeh[v_1,v_2]$.
\item[$\mathrm{(b)}$]  $\frakeh[v_1,\SS(v_2)]\simeq\frakh[v_1,\SS(v_2)]\times\bbP^{l-1}.$
\item[$\mathrm{(c)}$]  $\frakeh[v_1,\SS(v_2)]$ has a unique top dimensional component,
which is of dimension $\mid/2$.
\item[$\mathrm{(d)}$] $\bfr[v,0\,;\,v_1,v_2]|_{\frakeh[v_1,\SS(v_2)]}$ is a non zero multiple of the fundamental class
of the top dimensional irreducible component.
\end{itemize}
\end{proposition}

\subsubsection{Proof of $\mathrm{(a)}$, $\mathrm{(b)}$, $\mathrm{(c)}$}
Part (a) is easy. 
For (b), note that if a tuple $z_2=(\bar x_2,\bar a_2)$ in $\M_s(v_2,\delta_i)$ represents a point of the stratum
$\frakM_0(\kappa_{v_2,0})$, then, since $a_2^*\neq 0$ and since the representation $z_2$
must have a constituent of dimension $\delta_i$, we deduce that $a_2=0$.
Now, let us prove part (c).
For each $\tau\in RT(v,w)$ and $\tau_1\in RT(v_1,w)$ we write
\begin{align*}
\frakh(\tau\,;\,\tau_1)&=\frakh[v_1,\SS(v_2)]\,\cap\,(\frakM(\tau)\times \frakM(\tau_1)),\\
\frakeh(\tau\,;\,\tau_1)&=\frakeh[v_1,\SS(v_2)]\,\cap\,(\frakM(\tau)\times \frakM(\tau_1)\times\frakM(v_2,\delta_i)).\end{align*}
Part (c) follows from the following lemma.

\smallskip

\begin{lemma}\hfill
\begin{itemize}[leftmargin=8mm]
\item[$\mathrm{(a)}$] $\frakeh(\tau\,;\,\tau_1)\neq\emptyset\Rightarrow\tau \leqslant \tau_1\oplus\kappa_{v_2,0}.$
\item[$\mathrm{(b)}$] $\dim\frakeh(\tau\,;\,\tau_1)\leqslant \mid/2$.
\item[$\mathrm{(c)}$] $\dim\frakeh(\tau\,;\,\tau_1)=\mid/2\iff\tau=\kappa_{v_1,w} \oplus \kappa_{v_2,0}$ 
and $\tau_1=\kappa_{v_1,w}$.
\item[$\mathrm{(d)}$]   $\frakeh(\kappa_{v_1,w}\oplus\kappa_{v_2,0}\,;\,\kappa_{v_1,w})$ 
has a unique top-dimensional  irreducible component.

\end{itemize}
\end{lemma}
\begin{proof} Part (a) is obvious. We'll prove (b), (c), (d) simultaneously.
The proof goes along lines similar to that of \cite[prop.~3.16]{SV17a}. 
Consider the natural projection
$$p : \H^\circ[v_1,v_2] \to \M_s(v_1,w) \times \M(v_2).$$

\smallskip
 
First, assume that $v_2$ is not  the dimension vector of any constituent of $\tau_1$. 
Hence, we have $\frakeh(\tau\,;\,\tau_1)\neq\emptyset$ if and only if
$\tau= \tau_1\oplus\kappa_{v_2,0}.$ Then, we have 
$$\Hom_{\tilde\Pi}(N_1,N_2)=\{0\},\quad\forall (N_1,N_2)\in\M_s(\tau_1) \times \M(\kappa_{v_2,0}).$$
Here we view $N_1$ and $N_2$ as $\tilde\Pi$-modules as in \cite[\S 3.7.2]{SV17a}. 
By \cite[lem.~3.19]{SV17a}, the variety
$p^{-1}(\M_s(\tau_1) \times \M(\kappa_{v_2,0}))$ is smooth over $\M_s(\tau_1) \times \M(\kappa_{v_2,0})$ 
with connected 
fibers of dimension 
$$v_1 \cdot v_2-(v_1+\delta_\infty\,,\,v_2)_{\tilde{Q}}.$$ 
Fix an $I$-graded subspace $V_1 \subset V$ of dimension $v_1$ and let $P \subset G(v)$ be its stabilizer. 
The $P$-action on $\H^\circ[v_1,v_2]$ is free.
Therefore, we have
$$\dim\frakh^{\circ}(\tau_1\oplus\kappa_{v_2,0}\,;\,\tau_1) = 
-(v_1,v_2)_Q + w \cdot v_2 + \dim\frakM(\tau_1) + 2(q_i-1)l^2+1.$$
A short computation yields
\begin{align}\label{calcul}
\mid/2 - \codim_{\frakM(v_1,w)}\frakM(\tau_1)=-(v_1,v_2)_Q + w \cdot v_2 + \dim\frakM(\tau_1) 
+ 2(q_i-1)l^2+l.
\end{align}
We deduce that 
$$\dim\frakh(\tau_1\oplus\kappa_{v_2,0}\,;\,\tau_1)=\mid/2 - \codim_{\frakM(v_1,w)}\frakM(\tau_1)-l+1.$$
Hence, Proposition \ref{prop:D2} (b) yields
\begin{align}\label{E:req1}
\begin{split}
\dim\frakeh(\tau_1\oplus\kappa_{v_2,0}\,;\,\tau_1)&= \mid/2 - \codim_{\frakM(v_1,w)}\frakM(\tau_1).
\end{split}
\end{align}
Further, the set $\frakeh(\tau_1\oplus\kappa_{v_2,0}\,;\,\tau_1)$ is irreducible,
because both $\M_s(\tau_1)$ and $\M(\kappa_{v_2,0})$ are irreducible.
This proves the lemma in this case.

\smallskip

Next let us assume that $\tau_1$ does contain the dimension vector $v_2$. Let
\begin{align*}
U&=\{((\bar{x}_1,\bar{a}_1), \bar{x}_2) \in \M_s(\tau_1) \times \M(\kappa_{v_2,0})\;;\; 
\Hom_{\tilde\Pi}((\bar{x}_1,\bar{a}_1), \bar{x}_2)=0\},\\
Z&=(\M_s(\tau_1) \times \M(\kappa_{v_2,0})) \setminus U.
\end{align*}
Note that $U$ is a dense open subset of $\M_s(\tau_1) \times \M(\kappa_{v_2,0})$. 
By the same argument as above,
\begin{equation}\label{E:req2}
\dim(p^{-1}(U)/\!\!/P \times \mathbb{P}^{l-1})=\mid/2-\codim_{\frakM(v_1,w)}\frakM(\tau_1)
\end{equation}
and $p^{-1}(U)/\!\!/P$ is again irreducible. 
From equations (\ref{E:req1}) and (\ref{E:req2}) 
we see that the lemma will be proved once we show that, for any $\tau_1$ as above, we have
\begin{equation}\label{E:nreq}
\dim(p^{-1}(Z)/\!\!/P \times \mathbb{P}^{l-1}) < \mid/2.
\end{equation}

\smallskip

We will treat the case when the dimension vector $v_2$ occurs once in $\tau_1$, with multiplicity $k \geqslant 1$. 
The case when the 
dimension vector $v_2$ occurs more than once is similar and is left to the reader. 

\smallskip

Then, we have $\dim\Hom_{\tilde\Pi}(N_1,N_2) \leqslant k$ for any $(N_1,N_2) \in Z$ and
\cite[lem.~3.19]{SV17a} yields
\begin{equation}\label{E:req3}
\dim(p^{-1}(Z) /\!\!/ P)  \leqslant \dim Z-\dim L -(v_1+\delta_\infty\,,\,v_2)_{\tilde{Q}} + k.
\end{equation}
There can be a nonzero morphism from $N_1 \in \M(\tau_1)$ to $N_2 \in \M(\kappa_{v_2,0})$ 
only if the simple constituent of
$N_1$ of dimension $v_2$ is isomorphic to $N_2$. Therefore, we have
$$\dim Z \leqslant \dim\M_s(\tau_1) + \dim\M(\kappa_{v_2,0}) - \dim\frakM_0(\kappa_{v_2,0})$$
and thus
\begin{equation}\label{E:req4}
\dim Z-\dim L \leqslant \dim\frakM(\tau_1) -1.
\end{equation}
Next, by \cite[(3.19)]{SV17a} we have 
$$\dim\frakM(\tau_1) \leqslant  \dim\frakM(v_1,w) /2 +  \dim\frakM_0(\tau_1)/2.$$
Therefore, we have
\begin{align}\label{E:req5}
\codim_{\frakM(v_1,w)} \frakM(\tau_1) \geqslant 
\dim\frakM(v_1,w)/2-\dim \frakM_0(\tau_1)/2. 
\end{align}
Using \eqref{calcul}, \eqref{E:req3}, \eqref{E:req4} and \eqref{E:req5}, 
we deduce that (\ref{E:nreq}) is implied by the following inequality
\begin{equation}\label{E:nreq1}
1 + 2(q_i-1)l^2  + \dim\frakM(v_1,w)/2-\dim\frakM_0(\tau_1)/2 \geqslant k.
\end{equation}
Finally, we prove (\ref{E:nreq1}). It is obvious if $k=1$, so we may assume that $k>1$. 
Let $\tau'_1$ be the representation type
obtained from $\tau_1$ by replacing $(k,v_2)$ by $(1, v_2), (1,v_2), \ldots$ ($k$ times).
We have 
$$\dim \frakM_0(\tau'_1) = (k-1) (2 + 2(q_i-1)l^2) + \dim\frakM_0(\tau_1).$$
Since $q_i>1$, this implies that
$$1 + 2(q_i-1)l^2 + \dim \frakM_0(\tau'_1)/2-\dim\frakM_0(\tau_1)/2 >k.$$
Since $\tau_1$ and $\tau'_1$ have the same dimension type, they are either both in the image of $\pi$ or both outside. In the latter case, the set
$\frakeh(\tau; \tau_1)$ is empty and there is nothing to prove. In the former case, we have
$\dim\frakM(v_1,w) \geqslant \dim\frakM_0(\tau'_1)$
and the inequality (\ref{E:nreq1}) follows.

\end{proof}

\smallskip

\subsubsection{Proof of $\mathrm{(d)}$} We concentrate now on the proof of Proposition~\ref{prop:D2}(d).
For degree reasons and Propositions \ref{prop:D1}, \ref{prop:D2}(c), the restriction of the cycle
$\bfr[v,0\,;\,v_1,v_2]$ to $\frakeh[v_1,\SS(v_2)]$ is a multiple of the fundamental class of the top dimensional irreducible 
component.
We must prove it is not zero. We'll abbreviate
$$\frakL=\frakL(v,w+\delta_i),\quad\frakM=\frakM(v,w+\delta_i),\quad\frakM_0=\frakM_0(v,w+\delta_i),
\quad\frakZ=\frakZ(v,w+\delta_i).$$
Then, for each $k=0,1,\dots,l$ we consider the (reduced) fiber product
$$
\frakR_k=\frakL[v_1+k\delta_i,v_2-k\delta_i]\times_{\frakM_0}\frakM[v_1,v_2]
\subset \frakZ,
$$
and the closure $\overline\frakR$ in $\frakZ$ of the subset of $\frakR_0$ given by
$$\frakR=\frakL[v_1,v_2]\times_{\frakM[v_1,v_2]}\frakM[v_1,v_2].$$
Then $\overline\frakR$, $\overline\frakR_1,\dots,\overline\frakR_l$ are Lagrangian cycles in 
$\frakM^\op\times\frakM[v_1,v_2]$ such that
$\overline\frakR_k\subset\frakR_k\cup\cdots\cup\frakR_l$ for each $k$.
By \cite[prop.~3.5.1]{MO} we have
$$\stab[v_1,v_2]=C_0+C_1+\cdots +C_l$$ where
$C_0=\varepsilon_{v_1,v_2}\,[\,\overline\frakR\,]$
and $C_k$ is a Lagrangian cycle in $\frakM^\op\times\frakM[v_1,v_2]$
supported on $\overline \frakR_k$ for each $k$.
By Proposition \ref{prop:D1},
the restriction of the cycle $\stab[v_1,v_2]$ to $\frakM[v,0\,;\,v_1,v_2]$
is supported on the generalized Hecke correspondence
$$\frakeh[v_1,v_2]\subseteq \frakM[v,0]\times_{\frakM_0}\frakM[v_1,v_2].$$

\smallskip

\begin{lemma}\label{lem:AAA}
If $0<k<l$ then $\dim\big( \overline\frakR_k\cap\frakeh[v_1,\SS(v_2)]\big)<\mid/2$.
\end{lemma}

\begin{proof}
We'll prove that $\dim\big(\frakR_k\cap\frakeh[v_1,\SS(v_2)]\big)<\mid/2.$
Consider the constructible subset 
$X_k \subset \frakh[v_1,\SS(v_2)]$ of pairs $(z,z_1)$ for which there exist simple 
$\tilde\Pi$-modules $S_k,$ $S_l$ of 
respective dimension $k\delta_i,$ $l\delta_i$ such that 
$ \Hom_{\tilde\Pi} (z,S_k) \neq  0 $ and $\Hom_{\tilde\Pi} (z,S_l)  \neq 0$. 
We have 
$$(z, z_1,z_2) \in \frakR_{l-k} \cap\frakeh[v_1,\SS(v_2)]\Rightarrow (z,z_1)\in X_k.$$ 
So, we must show that
\begin{equation}\label{E:AAA1}
\text{codim}_{\frakh^{\circ}[v_1,v_2]}X_k >0.
\end{equation}
Let $Y \subset \frakh[v_1,\SS(v_2)]$ be the dense open subset of representations $(z,z_1)$ for 
which all non-rigid simple factors of $z$ occur with multiplicity one. It is enough to prove that
\begin{equation}\label{E:AAA2}
\text{codim}_{Y}X_k \cap Y >0 .
\end{equation}

\smallskip

Let $(z,z_1) \in X_k\cap Y$ and let $S_k,$ $S_l$ be the corresponding simple $\tilde{\Pi}$-modules. 
There is a short exact sequence
\begin{equation}\label{E:AAA3}
0 \to z_1 \to z \to S_l \to 0
\end{equation}
and since $\Hom(z,S_k) \neq 0$ and $\Hom(S_l,S_k) =0$ there is an exact sequence
\begin{equation}\label{E:AAA4}
0 \to z_2 \to z_1 \to S_k \to 0.
\end{equation}
From the exact sequence
\begin{equation}\label{E:AAA5}
\xymatrix{0 \ar[r] & \Hom(z,S_k) \ar[r] & \Hom(z_1,S_k) \ar[r]^-{\partial} & \Ext^1(S_l,S_k) }
\end{equation}
and the fact that $\Hom(z,S_k)=\Hom(z_1,S_k)=\mathbb{C}$ we deduce that $\Im(\partial)=0$. 
Hence, the following exact sequence splits
\begin{equation}\label{E:AAA6}
0 \to S_k \to z/z_2 \to S_l \to 0 .
\end{equation}

\smallskip

Consider the set 
$$\frakh \subset \frakM(v,w) \times \frakM(v_1,w) \times \frakM(v_1-k\delta_i,w)$$
consisting of the triples $(z,z_1,z_2)$ of stable $\tilde{\Pi}$-modules
such that $z_2 \subset z_1 \subset z$. 
The restriction of the map 
$$\rho: \frakh\to \frakh[v_1,v_2], \ 
(z , z_1 , z_2) \mapsto (z , z_1)$$ 
to $\rho^{-1}(Y)$ is finite. Indeed, any simple $\tilde{\Pi}$-module $S_k$ of dimension $k\delta_i$ is non-rigid, because
$$\dim\Ext^1_{\tilde \Pi}(S_k,S_k)=2-(k\delta_i,k\delta_i)_{\tilde Q}=2+2k^2(q_i-1)>0$$
by \cite[prop.~�3.1]{SV17a}.
Therefore, for any pair $(z, z_1) \in Y$ the $\tilde{\Pi}$-module $z_1$ only has finitely 
many simple factors of dimension $k\delta_i$ and they all occur with multiplicity one.
Thus, Lemma~\ref{lem:AAA} will be proved once we show that there exists $Z \subset \rho^{-1}(Y)$ such that 
\begin{equation}\label{E:AAA7}
\rho(Z) \supset X_k\cap Y\ \text{and}\  \text{codim}_{\rho^{-1}(Y)} Z >0.
\end{equation}
Let $P \subset G(k\delta_i+l\delta_i)$ be the standard parabolic subgroup of type $(k\delta_i, l\delta_i)$ and let 
$\frakp$ be its Lie algebra. Consider the stack
$$\frakH= \M(k\delta_i+l\delta_i) \cap \frakp^{2q_i} / P.$$ 
Let $U$ be the open substack of $\frakH\times\frakM(v_1-k\delta_i, w)$ consisting of 
triples $(x,y,z_2)$ such that all non-rigid simples occuring in $z_2 \oplus x \oplus y/x$ have multiplicity one. 
The set $\rho^{-1}(Y)$ is the open subset of $\frakh$ consisting of triples $(z ,z_1,z_2)$ for 
which all non-rigid simples occuring in $z$ have multiplicity one. 
As all simples of dimension in $\mathbb{N}\delta_i$ are non-rigid, an argument in all points parallel to 
\cite[lem.~3.19]{SV17a} shows that the map
$$\kappa~:\rho^{-1}(Y)\to U, \qquad (z ,z_1, z_2) \mapsto (z_1/z_2, z/z_2, z_2)$$
is a stack vector bundle. Moreover, the constructible substack $\frakH^{s}$ parametrizing pairs 
$(S_k \subset S_k \oplus S_l)$ with $S_k, $ $S_l$ simple, is of strictly positive codimension, because
$$\dim\;\Ext^1(S_l,S_k)=-(l\delta_i,k\delta_i)=2(q_i-1)kl >0.$$ 
Therefore, the following set satisfies the conditions in \eqref{E:AAA7}
$$Z=\kappa^{-1}\left((\frakH^{s}\times\frakM(v_1-k\delta_i, w)) \cap U\right).$$

\end{proof}

\smallskip

Now, let us come back to the proof of Proposition~\ref{prop:D2}(d).
From the lemma, we deduce that the restrictions of the cycles $C_1, C_2, \dots,C_{l-1}$ to $\frakM[v,0\,;\,v_1,v_2]$
do not contribute to $\bfr[v,0\,;\,v_1,v_2]$.
Further, by \eqref{res} we have 
$$\stab[v_1,v_2]|_{\frakM[v,0\,;\,v_1,v_2]}\in H_A^{\leqslant 2lv_i-2}\cap H_{*}^{\G/A}\big(\frakM[v,0\,;\,v_1,v_2]\big).$$
Let
$Z$ denote the restriction of the cycle $C_0+C_l$ to $\frakM[v,0\,;\,v_1,v_2]$.
We deduce that 
\begin{align}\label{Z}
Z\in H_A^{\leqslant 2lv_i-2}\cap H_{*}^{\G/A}\big(\frakM[v,0\,;\,v_1,v_2]\big)
\end{align}
and the cycle $\bfr[v,0\,;\,v_1,v_2]$ is the symbol of the class $Z$ in
$$H_A^{2lv_i-2}\otimes H_{*}^{\G/A}\big(\frakM[v,0\,;\,v_1,v_2]\big).$$
Further, the class $Z$ is supported on the generalized Hecke correspondence $\frakeh[v_1,v_2]$ and
it is enough to prove that its restriction to the open set
$\frakeh[v_1,\SS(v_2)]$ belongs to
$$H_A^{\leqslant\,2lv_i-2}\cap H_*^{\G/A}(\frakeh[v_1,\SS(v_2)])
\setminus
H_A^{<\,2lv_i-2}\cap H_*^{\G/A}(\frakeh[v_1,\SS(v_2)]).$$

\smallskip

For a future use, we fix some notation. 
Let $\fraktop[v_1,v_2]$ denote the top dimensional irreducible component of $\frakeh[v_1,\SS(v_2)]$ and
$N$ be the restriction to $\frakeh[v_1,v_2]$
of the normal bundle to the closed embedding of $\frakM[v,0\,;\,v_1,v_2]$ in $\frakM\times\frakM[v_1,v_2].$
Write $P$ for the standard parabolic subgroup of $G(v)$ of type $(v_1,v_2)$ and 
$L=G(v_1)\times G(v_2)$ for the corresponding standard Levi subgroup.
Finally, we abbreviate $E=\Hom_I(v_2,w_2)$.

\smallskip

First, we concentrate on the cycle $C_l$. 
The variety $\frakR_l$ is a closed subset of $\frakM\times\frakM[v_1,v_2]$ and an affine space bundle over
the Lagrangian subvariety
$$\frakM[v,0]\times_{\frakM_0}\frakM[v_1,v_2]\subset\frakM[v,0]^\op\times\frakM[v_1,v_2].$$
By Proposition \ref{prop:D1}, we can assume that the cycle $C_l$ is supported in $\frakR_l\cap\frakep[v_1,v_2]$.
The obvious map $\frakR_l\to\frakM[v,0\,;\,v_1,v_2]$ yields a 
$\G$-equivariant affine space bundle 
\begin{align}\label{nu}\frakR_l\cap\frakep[v_1,v_2]\to\frakeh[v_1,v_2]\end{align}
and the cycle $C_l$ is the pull-back of a Lagrangian cycle 
of $\frakM[v,0]^\op\times\frakM[v_1,v_2]$ supported in $\frakeh[v_1,v_2]$.
Therefore, the cycle 
$$Z_l=C_l|_{\frakM[v,0\,;\,v_1,v_2]}$$
is supported on $\frakeh[v_1,v_2]$ and its restriction to the open subset $\frakeh[v_1,\SS(v_2)]$
 is a rational multiple of the class
\begin{align}\label{Cl}
\eu(N/\nu_l)\cap [\fraktop[v_1,v_2]\,].
\end{align}
Here $\nu_l$ is the $\G$-equivariant vector subbundle of $N$ equal to the relative tangent bundle of \eqref{nu}.
Set $B_l=\GH[v_1,v_2]\times E^*$. Then, from \eqref{L} we deduce that $\nu_l$ is first projection
\begin{align*}
\big(B_l\times\Hom_I(w_2,v_1)\big)/\!\!/P\times L
\to
\GH[v_1,v_2]/\!\!/P\times L
\end{align*}
where 
$\G\times P$ acts on $\Hom_I(w_2,v)$ in the obvious way and $L$ acts trivially on $\Hom_I(w_2,v).$

\smallskip

Next, we consider the cycle $C_0$.
We have an inclusion
$\overline\frakR\subseteq\frakep[v_1,v_2]$ and $C_0$ is a multiple of the fundamental class of $\overline\frakR$.
Hence, the cycle 
$$Z_0=C_0|_{\frakM[v,0\,;\,v_1,v_2]}$$ is supported on $\frakeh[v_1,v_2]$.
Set
$$B_0=\{(z,z'_1,z'_2,\bar a_2)\in\GH[v_1,v_2]\times T^*E\,;\,
(a_2')^*(a_2)=0,\, a_2^*\in\bbC(a'_2)^*\}$$
with
$$z'_2=(\bar x'_2, \bar a'_2),\quad\bar a'_2=(a'_2,(a'_2)^*)\quad\bar a_2=(a_2,a_2^*).$$
We define the $\G$-equivariant vector subbundle $\nu_0\subset N$ as the obvious projection 
\begin{align*}
\big(B_0\times\Hom_I(w_2,v_1)\big)/\!\!/P\times L\to \frakeh[v_1,v_2].
\end{align*}
Consider the open subsets of $\frakep[v_1,v_2]$ and $\frakeh[v_1,v_2]$ given by
\begin{align}\label{sub}
\begin{split}
\frakep[v_1,v_2]^\circ&=\{(z,z'_1,z'_2)\in\GP[v_1,v_2]\,;\,\bar x_2\ \text{is\ simple},\ \Hom_{\tilde\Pi}(z_1,z_2)=0\}/\!\!/P\times L,\\
\frakeh[v_1,v_2]^\circ&=\{(z,z'_1,z'_2)\in\GH[v_1,v_2]\,;\,\bar x_2\ \text{is\ simple},\ \Hom_{\tilde\Pi}(z_1,z_2)=0\}/\!\!/P\times L,
\end{split}
\end{align}
where $z=(\bar x,\bar a)$, $z_2=(\bar x_2,\bar a_2)$ and
$$z_1=z|_{V_1\oplus W_1},\quad z_2=z\ \text{mod}\ V_1\oplus W_1.$$ 
Note that we have
\begin{align*}
\begin{split}
\frakeh[v_1,v_2]^\circ&=\frakep[v_1,v_2]^\circ\,\cap\,\frakM[v,0\,;\,v_1,v_2].
\end{split}
\end{align*}
We claim that we have
$\frakeh[v_1,S(v_2)]\subseteq\overline\frakR$ and that
the restriction of the fundamental class of $\overline\frakR\cap\frakep[v_1,v_2]^\circ$ to 
$\frakeh[v_1,v_2]^\circ$
is equal to
\begin{align}\label{C0}
\eu(N/\nu_0)\cap [\,\fraktop[v_1,v_2]^\circ\,],\quad
\fraktop[v_1,v_2]^\circ=\fraktop[v_1,v_2]\cap\frakeh[v_1,v_2]^\circ.
\end{align}

\smallskip

Comparing \eqref{Cl} and \eqref{C0}, we deduce that 
the restriction of the cycle $Z$ to the open subset $\frakeh[v_1,v_2]^\circ$ is of the form
$\alpha\cap [\,\fraktop[v_1,v_2]^\circ\,],$ for some class
$\alpha\in H^{2lv_i}_\G(\frakeh[v_1,v_2])$
which is a $\bbQ$-linear combination of the equivariant Euler classes $\eu(N/\nu_0)$ and $\eu(N/\nu_l)$.

\smallskip

Now, let $\frakL$ be the $\G$-equivariant 
line bundle on $\frakeh[v_1,v_2]$ whose fiber at the point represented by the tuple 
$(z,z'_1,z'_2)\in\GH[v_1,v_2]$ is the line in $E$ spanned by $(a'_2)^*$. 
The obvious projection 
$T^*E\to E^*$ yields a $\G$-equivariant vector bundle homomorphism $\nu_0\to\nu_l$ which fits in an exact sequence
\begin{align*}0\to\frakL\to\nu_0\to\nu_l\to\hbar\otimes\frakL^{-1}\to 0.\end{align*}
From \eqref{Z} we deduce that $\alpha$ is a non zero rational multiple of
the class $\hbar\,\eu(N/\nu_0+\nu_l).$
To finish the proof, it is enough to observe that the restriction 
of the class $\eu(N/\nu_0+\nu_l)$ to
$H^{2lv_i-2}_A(\frakeh[v_1,v_2])$ is non zero.

\smallskip

To prove the claim, set 
\begin{align*}
F&=\{(\bar a_2,\bar a'_2)\in T^*E\times T^*(E\setminus\{0\})\,;\,a_2a_2^*=a'_2(a'_2)^*\}.
\end{align*}
By \cite[lem.~�3.19]{SV17a}, the assignment $f:(z,z'_1,z'_2)\mapsto (\bar a_2,\bar a'_2)$
gives rise to the following commutative diagram
$$\xymatrix{
\frakep[v_1,v_2]^\circ\ar[r]^-f&F\\
\frakeh[v_1,v_2]^\circ\ar[r]^-{f'}\ar@{^{(}->}[u]& \{0\}\times E\setminus\{0\}.\ar@{^{(}->}[u]
}$$
whose horizontal maps are smooth and $P\times L$-equivariant.
We have
$$f(\frakR\cap\frakep[v_1,v_2]^\circ)=
\{(\bar a_2,\bar a'_2)\in T^*(E\setminus\{0\})\times T^*(E\setminus\{0\})\,;\,\exists\, t\in\bbC,\,a'_2=ta_2,\,a_2^*=t(a'_2)^*\}.$$
Hence the map $f$ restricts to a smooth map 
$$\overline\frakR\cap\frakep[v_1,v_2]^\circ\to\{(\bar a_2,\bar a'_2)\in T^*E\times T^*(E\setminus\{0\})\,;\,
(a_2')^*(a_2)=0,\, a_2^*\in\bbC(a'_2)^*\}.$$
Comparing the fibers of $f$ and $f'$ gives the result.

\qed

\smallskip

\subsection{Proof of Proposition \ref{prop:key}}\hfill\\

Recall that $w_1=w$, $w_2=\delta_i$ and $v=v_1+v_2$ with $v_2=l\delta_i$ and $l$ a positive integer.

\subsubsection{Proof of Proposition $\ref{prop:key}$}

First, assume that $q_i>1$. 
Consider the irreducible component of the Lagrangian quiver variety $\frakL^1(v_2,\delta_i)$ given by
\begin{align*}\scrC
=\{[\bar x,\bar a]\in\frakM(v_2,\delta_i)\,;\,x=a=0\}.
\end{align*}
Note that
$$\scrC=\{(\bar x,\bar a)\in\R_s(v_2,\delta_i)\,;\,x=a=0\}/\!\!/G(v_2),$$
hence it is smooth of dimension $d_{v_2,\delta_i}/2$.
Let $\scrC_0$ be the image of $\scrC$ by $\pi$.
The closed embedding
$i:\scrC\to\frakM(v_2,\delta_i)$ and the projection
$\pi:\scrC\to \scrC_0$ yield the map
$$\phi=\pi_*\circ i^*:H^{T}_*(\frakM(v_2,\delta_i))\to H_{*-d_{v_2,\delta_i}}^{T}(\scrC_0).$$ 
Since the torus $T$ contains a one parameter subgroup which scales all the quiver 
data by the same scalar, the $T$-fixed points locus in $\scrC_0$ is $\{0\}$.
Therefore,  the pushforward by the closed embedding $\{0\}\to\scrC_0$ 
is an isomorphism
$K\to H_*^{T}\big(\scrC_0\big)\otimes K.$
Let $\psi$ be the inverse map.
The composed map $\psi\circ\phi$
belongs to $F_{i}(v_2)^\vee\otimes K$.
The same argument as in \cite[prop.~5.2]{SV17a} implies that
$F_i(v_2)\otimes K$ is $K$ times the $T$-equivariant homology
of a smooth projective variety, hence the pushforward to a point yields a nondegenerate pairing
$$(\bullet,\bullet):F_i(v_2)\otimes F_i(v_2)\otimes K\to F_i(0)\otimes K=K.$$ 
Write
$\psi\circ\phi=(m,\bullet)$ for some element 
\begin{align}\label{m}m\in\Hom(F_i(0),F_i(v_2))\otimes K\subset A_i^f(v_2)\otimes K.
\end{align}
The restriction of the $K$-linear map $\e(m\,;\,w)$ to $F_w(v_1)\otimes K$ is an operator
$$\e(v_1,m\,;\,w)\in\Hom_{\Bbbk[w]}(F_w(v_1),F_w(v))\otimes K.$$
For each composition $\nu$ of $v_2$, the
fundamental class of the Hecke correspondence
$\frakh[v_1,\Lambda_\nu\,;\,w]$ can be viewed as a $\Bbbk[w]$-linear operator
$F_w(v_1)\to F_w(v)$ by convolution.

\smallskip

\begin{lemma}\label{lem:59}
We have the following equality in $\Hom_{\Bbbk[w]}(F_w(v_1),F_w(v))\otimes K$
\begin{align*}\e(v_1,m\,;\,w)=
\sum_{\nu\,\vDash\, v_2}a_\nu(v,w)\,[\,\frakh[v_1,\Lambda_\nu\,;\,w]\,],\quad a_\nu(v,w)\in\bbQ.
\end{align*}
\end{lemma}

\begin{proof}
Under the Kunneth isomorphism, the maps $\phi$ and $\psi$ give linear maps
\begin{align*}
\phi&:H_*^{G(w)\times T}\big(\frakM[v,0\,;\,v_1,v_2]\big)\to 
H_{*-d_{v_2,\delta_i}}^{G(w)\times T}\big(\frakM(v,w)\times\frakM(v_1,w)\times \scrC_0\big),\\
\psi&:H_*^{G(w)\times T}\big(\frakM(v,w)\times\frakM(v_1,w)\times \scrC_0\big)\otimes K \to
H_*^{G(w)\times T}\big(\frakM(v,w)\times\frakM(v_1,w)\big)\otimes K.
\end{align*}
By Proposition \ref{prop:D1}, the class 
$\phi(\bfr[v,0\,;\,v_1,v_2])$  is supported on the variety $\frakeh[v_1,v_2\,;\,w,\scrC_0]$.
The projection along the third component
$\frakM(v,w)\times\frakM(v_1,w)\times \scrC_0\to\frakM(v,w)\times\frakM(v_1,w)$
gives an isomorphism
\begin{align*}\frakeh[v_1,v_2\,;\,w,\scrC_0]=\frakeh[v_1,\Lambda^1(v_2)\,;\,w,\scrC_0]\simeq\frakh[v_1,\Lambda^1(v_2)\,;\,w]
.\end{align*}
By \eqref{IrrLambda} and Remark \ref{rem:3.6}, 
the set of irreducible components of $\frakeh[v_1,v_2\,;\,w,\scrC_0]$ is
$$\big\{\frakeh[v_1,\Lambda_\nu\,;\,w,\scrC_0]\,;\,\nu\vDash v_2\big\}.$$
Further, the variety $\frakeh[v_1,v_2\,;\,w,\scrC_0]$ is pure dimensional 
of dimension 
$$(\mid-d_{v_2,\delta_i})/2=(d_{v,w}+d_{v_1,w})/2.$$
Since the cycle 
$\bfr[v,0\,;\,v_1,v_2]$ has the degree equal to $\mid$, the class 
$\phi(\bfr[v,0\,;\,v_1,v_2])$ has the degree equal to $\mid-d_{v_2,\delta_i}$.
So, there are rational numbers $a_\nu(v,w)$ such that
\begin{align}\label{phi(stab)}\phi(\bfr[v,0\,;\,v_1,v_2])=
\sum_{\nu\,\vDash\, v_2}a_\nu(v,w)\,[\,\frakeh[v_1,\Lambda_\nu\,;\,w,\scrC_0]\,],
\end{align}
hence we have
\begin{align*}
\psi\phi(\bfr[v,0\,;\,v_1,v_2])=
\sum_{\nu\,\vDash\, v_2}a_\nu(v,w)\,[\,\frakh[v_1,\Lambda_\nu\,;\,w]\,].
\end{align*}

\end{proof}

\smallskip

\noindent Now, Proposition \ref{prop:key} in the case $q_i>1$ follows by induction from 
\cite[prop.~3.22]{SV17a} and the next proposition.

\smallskip

\begin{proposition} \label{lem:5.15} We have 
\begin{itemize}[leftmargin=8mm]
\item[$\mathrm{(a)}$] $a_{(v_2)}(v,w)\neq 0$ for each $v,w$,
\item[$\mathrm{(b)}$] $a_\nu(v,w)$ does not depend on $v,$ $w$ for each $\nu$.
\end{itemize}
\end{proposition}

\smallskip

Finally, assume that $q_i\leqslant 1$. By Proposition \ref{prop:A}, we may assume that $l=1$. 
We have 
$$\frakC(v_1,v_2\,;\,w)=\frakh[v_1,\Lambda_{(v_2)}\,;\,w].$$
Hence, we can define the element $m$ as in \eqref{m}.
Then, the same argument as in Lemma \ref{lem:59} and Proposition \ref{lem:5.15} yields
$$\e(v_1,m\,;\,w)=[\,\frakC(v_1,v_2\,;\,w)\,].$$ 
So the convolution with the correspondence $\frakC_{i,1}$ belongs to $\e\big(A_{\delta_i}(\delta_i)\big)$.
Proposition \ref{prop:key} is proved.

\smallskip

\subsubsection{Proof of Proposition $\ref{lem:5.15}\mathrm{(a)}$} 
Write
\begin{align*}
\scrC^\circ=\{[\bar x,\bar a]\in\scrC\,;\,\bar x\in \SS(v_2)\},\quad
\scrC_0^\circ=\pi(\scrC^\circ).
\end{align*}
Note that we have
$$\frakeh[v_1,\SS(v_2)\,;\,w,\scrC]=
\frakeh[v_1,v_2\,;\,w,\scrC^\circ].$$
Hence, we have the following fiber diagram 
$$\xymatrix{
\frakM(v,w)\times\frakM(v_1,w)\times\scrC_0&
\frakM(v,w)\times\frakM(v_1,w)\times\scrC\ar[r]^-i\ar[l]_-\pi&
\frakM[v,0\,;\,v_1,v_2]\\
&\frakeh[v_1,v_2\,;\,w,\scrC]\ar[r]^-{g}\ar@{^{(}->}[u]^-{\alpha}&
\frakeh[v_1,v_2]\ar@{^{(}->}[u]^-{\alpha}\\
\frakM(v,w)\times\frakM(v_1,w)\times\scrC_0^\circ&\frakeh[v_1,v_2\,;\,w,\scrC^\circ]\ar[r]^-{f}\ar@{^{(}->}[u]^-{\beta}\ar[l]_-\pi
&\frakeh[v_1,\SS(v_2)].\ar@{^{(}->}[u]^-{\beta}\\
}$$
Let $g^!$ and $f^!$ denote the refined pullback morphisms $(i,g)^!$ and $(i,f)^!$. 
See \cite[\S 2.3.5]{SV17a} for more details.
Since the cycle $\bfr[v,0\,;\,v_1,v_2]$ is supported on $\frakeh[v_1,v_2]$, the proper base change implies that
\begin{align*}i^*\alpha_*(\bfr[v,0\,;\,v_1,v_2])=\alpha_* g^!(\bfr[v,0\,;\,v_1,v_2]),\end{align*}
from which we deduce that
\begin{align*}\phi(\bfr[v,0\,;\,v_1,v_2])=\pi_*\alpha_*g^!(\bfr[v,0\,;\,v_1,v_2]).\end{align*}
So, the restriction of the class $\phi(\bfr[v,0\,;\,v_1,v_2])$ by the open embedding 
$$\frakM(v,w)\times\frakM(v_1,w)\times\scrC^\circ_0\subset\frakM(v,w)\times\frakM(v_1,w)\times\scrC_0$$
is the pushforward by $\pi$ of the class
\begin{align*}\beta^*g^!(\bfr[v,0\,;\,v_1,v_2])=f^{!} \beta^*(\bfr[v,0\,;\,v_1,v_2]).\end{align*}

\smallskip

First, let us compute the later.
The map $i$ is a regular embedding of codimension $d_{v_2,\delta_i}/2$.
Consider the open subset $\frakeh[v_1,v_2]^\circ\subset\frakeh[v_1,v_2]$ introduced in \eqref{sub}, and set
$$\frakeh[v_1,\Lambda_{(v_2)}\,;\,w,\scrC^\circ_0]^\circ=\frakeh[v_1,v_2]^\circ\cap\frakeh[v_1,\Lambda_{(v_2)}\,;\,w,\scrC^\circ_0].$$
Using \cite[lem.~�3.19]{SV17a} as above, we get a Cartesian square
$$\xymatrix{\frakeh[v_1,\Lambda_{(v_2)}\,;\,w,\scrC^\circ_0]^\circ\ar[r]^-f\ar[d]&\frakeh[v_1,v_2]^\circ\ar[d]\\
\frakM(v_1,w_1)\times\scrC^\circ\ar[r]&\frakM(v_1,w_1)\times\frakM(\tau_{v_2}),}$$
where the vertical maps are smooth and $\tau_{v_2}\in RT(v_2,\delta_i)$ is the representation type $(1,0,\delta_i;1,v_2)$.
We have
$$\dim\scrC^\circ=d_{v_2,\delta_i}/2=(q_i-1)l^2+l,\quad
\dim\frakM(\tau_{v_2})=2(q_i-1)l^2+l+1.$$
So, up to restricting the map $f$ to the open subsets above, we may assume that
it is a regular embedding of codimension $d_{v_2,\delta_i}/2-l+1$ and
the excess intersection formula yields
$$f^! \beta^*(\bfr[v,0\,;\,v_1,v_2])=\eu(N)\cap f^* \beta^*(\bfr[v,0\,;\,v_1,v_2])$$
for some vector bundle $N$ of rank $l-1$.
Note that $\Lambda_\nu\cap\SS(v_2)=\emptyset$ for each composition $\nu\neq(v_2)$.
Therefore, we have
\begin{align}\label{FF1}\frakeh[v_1,\Lambda_\nu\,;\,w,\scrC^\circ]=\emptyset,\quad\forall\nu\neq(v_2).\end{align}
We deduce that 
\begin{align}\label{F2}
f^!\beta^*(\bfr[v,0\,;\,v_1,v_2])=\eu(N)\cap 
[\,\frakeh[v_1,\Lambda_{(v_2)}\,;\,w,\scrC^\circ]\,].
\end{align}
Next, the argument in the proof of Proposition \ref{prop:D2} implies that there is an isomorphism 
$\scrC^\circ=\scrC^\circ_0\times\bbP^{\,l-1}$ 
which identifies the map $\pi:\scrC^\circ\to\scrC^\circ_0$ with the first projection.
We deduce that the pushforward of the class \eqref{F2}
by the map $\pi$
is a nonzero multiple of the fundamental class of
$\frakeh[v_1,\Lambda_{(v_2)}\,;\,w,\scrC^\circ_0].$

\smallskip

Now, from \eqref{FF1} we get that
$$\frakeh[v_1,\Lambda_\nu\,;\,w,\scrC^\circ_0]=\emptyset,\quad\forall\nu\neq(v_2).$$
Thus, by  \eqref{phi(stab)} the restriction of the class $\phi(\bfr[v,0\,;\,v_1,v_2])$ to 
$\frakM(v,w)\times\frakM(v_1,w)\times\scrC^\circ_0$ is 
$$a_{(v_2)}(v,w)\,[\,\frakeh[v_1,\Lambda_{(v_2)}\,;\,w,\scrC^\circ_0]\,].$$
We deduce that $a_{(v_2)}(v,w)\neq 0$.

\smallskip

\subsubsection{Proof of Proposition $\ref{lem:5.15}\mathrm{(b)}$, step 1} 
Assume that $v_1=0$ and $v_2=v$.
 Fix dimension vectors $w_1$ and $w_0$ with 
$w_1\geqslant w_0$ and set $N=v\cdot(w_1-w_0)$. 
Let us prove that 
\begin{align}\label{cas1}
a_\nu(v,w_1)=a_\nu(v,w_0),\quad\forall\nu\vDash v.
\end{align}
For each $\epsilon=0,1$ we write $\G_\epsilon=G(w_\epsilon)\times G(\delta_i)\times T$ and
\begin{align*}
\frakL_{\epsilon,{\bbA^1}}&=\frakL\,[0,v\,;\,w_\epsilon,\delta_i]_{\bbA^1},\quad\hfill
&\frakZ_{\epsilon,{\bbA^1}}&=\frakZ\,[0,v\,;\,w_\epsilon,\delta_i]_{\bbA^1},&\\
\frakh_{\epsilon,\bbA^1}&=\frakh\,[0,v\,;\,w_\epsilon]_{\bbA^1},\quad\hfill
&\frakeh_{\epsilon,\bbA^1}&=\frakeh\,[0,v\,;\,w_\epsilon,\delta_i]_{\bbA^1},&\\
\frakp_{\epsilon,\bbA^1}&=\frakp\,[0,v\,;\,w_\epsilon,\delta_i]_{\bbA^1}.&&\end{align*}
Hence, we have
\begin{align*}
\frakp_{\epsilon,\bbA^1}&=\P_{\epsilon,\bbA^1}\,/\!\!/\,G(v),\\
\P_{\epsilon,\bbA^1}&=\{(\bar x,\bar a)\in\M_s(v,w_\epsilon+\delta_i)_{\bbA^1}\,;\,a(W_\epsilon)=0\}.
\end{align*}
Fix a surjective morphism of $I$-graded vector spaces
\begin{align}\label{p}p:W_1\to W_0.\end{align}
The composition with $p$ yields an $\bbA^{N}$-torsor
$$p:\P_{1,\bbA^1}\to\M(v,w_0+\delta_i)_{\bbA^1}.$$
Consider the open subset of $\frakp_{1,\bbA^1}$ given by
\begin{align*}
\frakp^\diamond_{1,\bbA^1}=p^{-1}\big(\M_s(v,w_0+\delta_i)_{\bbA^1}\big)\,/\!\!/\,G(v).
\end{align*}
Then, the map $p$ yields an $\bbA^N$-torsor
$$p:\frakp_{1,\bbA^1}^\diamond\to\frakp_{0,\bbA^1}.$$
The map \eqref{inclusion} gives an inclusion $\frakL_{1,\bbA^1}\subset\frakp_{1,\bbA^1}^\diamond$.
Since
\begin{align*}
\frakZ_{\epsilon,\bbG_m}&=\frakL_{\epsilon,\bbG_m}\times_{\frakM(v,\delta_i)_{\bbG_m}}\frakM(v,\delta_i)_{\bbG_m},
\end{align*}
we get the following fiber diagram 
\begin{align*}
\begin{split}
\xymatrix{
\frakZ_{1,\bbG_m}\ar@{^{(}->}[r]\ar[d]^-p&\,\frakp^\diamond_{1,\bbA^1}\times\frakM(v,\delta_i)_{\bbA^1}\ar[d]^-p
&\,\frakp^\diamond_{1}\times\frakM(v,\delta_i)\ar[d]^-p\ar@{_{(}->}[l],\\
\frakZ_{0,\bbG_m}\ar@{^{(}->}[r]&\,\frakp_{0,\bbA^1}\times\frakM(v,\delta_i)_{\bbA^1}
&\,\frakp_{0}\times\frakM(v,\delta_i),\ar@{_{(}->}[l]
}
\end{split}
\end{align*}
where the vertical maps are $\bbA^{N}$-torsors and 
$\frakp_{\epsilon},$ $\frakp_{\epsilon}^\diamond$ are the fibers at 0.

\smallskip

Now , we have the following general fact.

\smallskip

\begin{lemma}\label{lem:spe}
Let $p: X_{\bbA^1}\to Y_{\bbA^1}$ be a $T$-equivariant smooth morphism of $\bbA^1$-schemes which
are $T$-equivariant 
locally trivial fibrations over $\bbG_m$.
Let $X$, $Y$ be the fibers at 0. 
For any $T$-equivariant cycle $Z$ in 
$Y_{\bbG_m}$ we have the equality of $T$ equivariant cycles in $X$
$$p^*\lim_{0} Z=\lim_{0} \,p^*(Z)$$
\end{lemma}

\smallskip

\begin{proof} Set $Z=\sum_in_i[Z_i]$ where $Z_i$ is a $T$-invariant closed subvariety of $Y_{\bbG_m}$
and $n_i\in\bbQ$. Then, taking the Zarisky closure in $Y_{\bbA^1}$ and in $X_{\bbA^1}$ we get
$$\overline Z=\sum_in_i[\overline Z_i],\quad\overline {p^*(Z)}=\sum_in_i[\overline {p^{-1}(Z_i)}].$$ 
We have
$$p^*\lim_{0} Z=p^*i^*(\overline {Z}),\quad 
\lim_{0} \,p^*(Z)=j^*\overline{p^*(Z)},$$
where $i$, $j$ are the regular embeddings $Y\subset Y_{\bbA^1}$ and $X\subset X_{\bbA^1}$.
The functoriality of Gysin morphisms with respect to smooth pull back yields $p^*\circ i^*=j^*\circ p^*$.
Thus, the claim follows from the smoothness of $p$ and the equality
$\overline {p^*(Z)}=p^*(\overline Z)$, see \cite[thm.~2.3.10]{EGAIV}.

\end{proof}

\smallskip

Let us abbreviate
$$\stab_\epsilon=\stab[0,v\,;\,w_\epsilon,\delta_i],\quad
\bfr_\epsilon=\bfr\,[v,0\,;\,0,v\,;\,w_\epsilon,\delta_i].$$ 
From \eqref{stab00} we deduce that 
\begin{align}\label{cycle1}
\begin{split}
\stab_\epsilon
=\pm\lim_{0}\,[\frakZ_{\epsilon,\bbG_m}].
\end{split}
\end{align} 
It is a $\G_\epsilon$-equivariant cycle 
supported on the subset of $\frakM(v,w_\epsilon+\delta_i)\times\frakM(v,\delta_i)$
given by
$$\frakep_\epsilon:=\frakp_{\epsilon}\times_{\frakM_0(v,\delta_i)}\frakM(v,\delta_i).$$ 
Let $\stab_1|_{\frakep_1^\diamond}$ be the restriction of $\stab_1$ to the open subset
$$\frakep_1^\diamond:=\frakp^\diamond_{1}\times_{\frakM_0(v,\delta_i)}\frakM(v,\delta_i).$$ 
By Lemma \ref{lem:spe} we have
\begin{align}\label{relation}\stab_1|_{\frakep_1^\diamond}=p^*(\stab_0).
\end{align}

\smallskip

Now, let us consider the $G(w_\epsilon)\times T$-equivariant cycle
$\bfr_\epsilon$. Let 
$$i_\epsilon:\frakM(v,w_\epsilon)\times\frakM(v,\delta_i)\to\frakM(v,w_\epsilon+\delta_i)\times\frakM(v,\delta_i)$$
be the obvious embedding. Recall that $\res_\epsilon$ is the residue along $i_\epsilon$ of the cycle 
$\stab_\epsilon$. By \eqref{hp} we have the following fiber diagram
\begin{align}\label{diag5}\begin{split}
\xymatrix{
\frakM(v,w_1)\times\frakM(v,\delta_i)\,\ar@{^{(}->}[r]^-{i_1}&\,\frakM(v,w_1+\delta_i)\times\frakM(v,\delta_i)\\
\frakeh^\diamond_1\,\ar@{^{(}->}[u]\ar@{^{(}->}[r]\ar[d]^-p&\,\frakep_1^\diamond\ar@{^{(}->}[u]
\ar[d]^-p\\
\frakeh_0\,\ar@{^{(}->}[r]&\,\frakep_0.
}
\end{split}\end{align}

The cycle $\bfr_\epsilon$ 
is supported on $\frakeh_\epsilon$ and it is characterized by the following relation
$$(i_\epsilon)^*(\stab_\epsilon)=\text{eu}\cap\bfr_\epsilon\ \text{modulo}\ H_{G(\delta_i)}^{<\,\top}\otimes
H_*^{G(w_\epsilon)\times T}(\frakeh_\epsilon),$$
where $\text{eu}$ is a square root of the Euler class of the normal bundle to $i_\epsilon$.
See \S\ref{sec:res} for more details.
Since the map $p$ is smooth, using this characterization and the formula \eqref{relation}, we deduce that
the restriction $\bfr_1|_{\frakeh^\diamond_1}$ of $\bfr_1$ to the open subset $\frakeh^\diamond_1$ of $\frakeh_1$
above is given by 
\begin{align}\label{r01}\bfr_1|_{\frakeh^\diamond_1}=p^*(\bfr_0).\end{align}
Note that $\bfr_1$ is a $G(w_1)\times T$-cycle while $\bfr_0$ is $G(w_0)\times T$-equivariant.
Hence we must first apply to $\bfr_1$ the functoriality of equivariant cohomology relatively to any embedding
$G(w_0)\subset G(w_1)$ which is compatible with the flag $W_1\to W_0$ in \eqref{p}.

\smallskip

Finally, by Lemma \ref{lem:59} the class
$\e(m,v,w_\epsilon)$ is supported on the closed subset $\frakh_\epsilon$ of $\frakM(v,w_\epsilon)$
and it decomposes as the sum
\begin{align*}\e(m,v,w_\epsilon)=\sum_{\nu\,\vDash\, v}a_\nu(v,w_\epsilon)\,
[\,\frakh[0,\Lambda_\nu\,;\,w_\epsilon]\,].
\end{align*}
From \eqref{hp} we may consider the open subset
$\frakh^\diamond_1=\frakh_1\cap\frakp_1^\diamond$ of $\frakh_1$.
From \eqref{r01} we deduce that the restriction of $\e(m,v,w_1)$ to 
$\frakh^\diamond_1$ is equal to $p^*(\e(m,v,w_0))$.  
This implies the claim \eqref{cas1}, because for all composition $\nu$ of $v$ we have
$$\frakh^\diamond_1\cap\frakh_1[0,\Lambda_\nu\,;\,w_1]\neq\emptyset.$$

\smallskip

\subsubsection{Proof of Proposition $\ref{lem:5.15}\mathrm{(b)}$, step 2} 
Fix dimension vectors $v$, $w$. Assume that $v=v_1+v_2$ with $v_2 =l\delta_i$. Let us prove that 
\begin{align}\label{cas2}a_\nu(v,w)=a_\nu(v_2,w),\quad \forall\nu\vDash v_2.
\end{align}
Set $v_0=0$. For each $\epsilon=0,1$  we write 
\begin{align*}
\frakL_{\epsilon,\bbA^1}&=\frakL\,[v_\epsilon,v_2\,;\,w,\delta_i]_{\bbA^1},\quad
&\frakZ_{\epsilon,\bbA^1}&=\frakZ\,[v_\epsilon,v_2\,;\,w,\delta_i]_{\bbA^1},&\\
\frakh_\epsilon&=\frakh\,[v_\epsilon,v_2\,;\,w],\quad&\frakeh_\epsilon&=\frakh\,[v_\epsilon,v_2\,;\,w,\delta_i],&\\
\frakp_{\epsilon,\bbA^1}&=\frakp\,[v_\epsilon,v_2\,;\,w,\delta_i]_{\bbA^1},&\quad
\P_{\epsilon,\bbA^1}&=\P\,[v_\epsilon,v_2\,;\,w,\delta_i]_{\bbA^1}.
\end{align*}
Fix $V_1,$ $V_2$, $W$ and $P$ as in the previous sections. Set 
\begin{align*}
\P'_{0,\bbA^1}&=\{z\in\M(v_2,w+\delta_i)_{\bbA^1}\,;\,z(W)=0\}
=\M(v_2,\delta_i)_{\mathbb{A}^1} \times \Hom_I(V_{2},W).
\end{align*}
Then, we have 
$\P_{0,\bbA^1}=\P'_{0,\bbA^1}\cap\, \M_s(v_2,w+\delta_i)_{\mathbb{A}^1}.$
There is an obvious map 
\begin{align*}
p\,:\,\P_{1,\bbA^1}\to \P'_{0,\bbA^1}.
\end{align*}
We define
$$\frakp^\triangledown_{1,\bbA^1}=p^{-1}\big(\P_{0,\bbA^1}\,\big)\,/\!\!/\,P.$$
Since $\frakp_{0,\mathbb{A}^1}$ is the categorical quotient of $\P_{0,\bbA^1}$ by $G(v_2)$,
the map $p$ factors to a map
\begin{align*}
p\,:\,\frakp^\triangledown_{1,\bbA^1}\to\frakp_{0,\bbA^1}.
\end{align*}
This map may not be smooth. To remedy this, we will restrict it to a suitable open subset
 of $\frakp^{\triangledown}_{1,\bbA^1}$. 
For each representations $z_1=(\bar x_1,\bar a_1)$ and $z_2=(\bar x_2,\bar a_2)$ we define
\begin{align*}
U(z_1,z_2)&=\{ \varphi \in \Hom_{\k\bar{Q}}(\bar{x}_1,\bar{x}_2)\;;\; a_2^*\circ\varphi=0\},\\
U&=\{(z_1,z_2) \in \frakM(v_1,w)_{\mathbb{A}^1} \times_{\bbA^1} (\M(v_2,\delta_i)_{\bbA^1}\,/\,G(v_2))\;;\; 
U(z_1,z_2)=0\}.
\end{align*}
So $U$ is non-empty and open. 
Consider the obvious map
$$q=(q_1,q_2)\,:\,\frakp_{\epsilon,\bbA^1}\to
\frakM(v_\epsilon,w)_{\mathbb{A}^1} \times_{\bbA^1} (\M(v_2,\delta_i)_{\bbA^1}\,/\,G(v_2)),\ 
z\mapsto(z_1,z_2).$$
Then, we have a non-empty open subset of $\frakp_{1,\bbA^1}$ given by
\begin{align*}
\frakp^\diamond_{1,\bbA^1}&=\{z\in\frakp^\triangledown_{1,\bbA^1}\,;\,q(z) \in U\}.
\end{align*} 
Let $\frakp^{\diamond}_1$, $\frakp_1$ be the fibers at 0 of $\frakp^{\diamond}_{1,\bbA^1}$, $\frakp_{1,\bbA^1}$.
Note that 
$\frakp^{\diamond}_{1}$ intersects all strata $\frakh(v_1,\Lambda_{\nu}\,;\,w)$ in $\frakp_{1}$. 
Indeed, for a generic element
$z \in \frakh(v_1,\Lambda_{\nu}\,;\,w)$ we have $\Hom_{\k\bar{Q}}(\bar{x}_1,\bar{x}_2)=0$, 
where $z_1=(\bar{x}_1, \bar{a}_1)=z|_{V_1 \oplus W}$ and $(\bar{x}_2)$ is the endomorphsim
of $V/V_1$ induced by $z$.

\smallskip

\begin{lemma}\hfill
\begin{itemize}[leftmargin=8mm]
\item[$\mathrm{(a)}$]
The obvious inclusion $\frakL_{1,\bbA^1}\subset\frakp_{1,\bbA^1}$ maps into $\frakp_{1,\bbA^1}^\diamond$.
\item[$\mathrm{(b)}$] 
The map $p\,:\,\frakp_{1,\bbA^1}^\diamond\to\frakp_{0,\bbA^1}$ is smooth.
\end{itemize}
\end{lemma}

\smallskip

\begin{proof}
The inclusion $\frakL_{1,\bbA^1}\subset\frakp^\triangledown_{1,\bbA^1}$ is obvious. 
Next, note that for each $z \in \frakp_{1,\bbA^1}$ such that
$z_2$ is stable, we have   
$q(z) \in U$. 
Indeed, the image of any $\varphi \in U(q(z))$ is a $\bar{x}_2$-stable subspace of $V_2$ 
which is contained in $\Ker(a_2^*)$. 
This implies the part (a). 
We now turn to (b). The map 
$p$ is the composition of the chain of morphisms
$$\xymatrix{
\frakp^{\triangledown}_{1,\bbA^1} \ar[r]^-{(q_1,p)}& \frakM(v_1,w)_{\bbA^1} \times_{\bbA^1} \frakp_{0,\bbA^1}
\ar[r]^-{pr_2}&\frakp_{0,\bbA^1},}$$
where $pr_2$ is the projection on the second factor. The map $pr_2$ is a smooth map because
$\frakM(v_1,w)_{\bbA^1} \to \bbA^1$ is smooth. We claim that the restriction of $(q_1,p)$ to
$\frakp^{\diamond}_{1,\bbA^1}$ is an affine fibration over its image. 
Indeed, the fiber of $(q_1,p)$ over a pair 
$$((\bar x_1,\bar a_1), (\bar x_2,\bar a_2,v)) \in \M_s(v_1,w)_{\bbA^1} \times_{\bbA^1} \P'_{0,\bbA^1}$$
 is identified with the zero set of the affine map
$$\mu~: \bigoplus_{h \in \bar{\Omega}} \Hom(V_{2,h'}, V_{1,h''}) \oplus \Hom(\k,V_{1,i}) \to \Hom(V_{2,i},V_{1,i})$$
defined by
$$(y,u) \mapsto \sum_h (x_{1,h}y_{h^*} + y_h x_{2,h^*}-y_{h^*}x_{2,h} - x_{1,h^*}y_h) + ua_{2,i}^* + a_{1,i}v.$$
Identifying $\Hom(V_{2,i},V_{1,i})^*$ with $\Hom(V_{1,i}, V_{2,i})$ via 
the trace pairing, it is easy to see that $$\Im(\mu-\mu(0))^\perp=U(z_1,z_2). $$
Hence $\mu$ is surjective whenever 
$U(z_1,z_2)=0$. It follows that the restriction of $(q_1,p)$ to $\frakp^{\diamond}_{1,\bbA^1}$ is an 
affine fibration over the open subset
$U'= \pi^{-1}(U)$
where 
$$\pi : \frakM(v_1,w)_{\bbA^1}\times_{\bbA^1}\frakp_{0,\bbA^1} \to 
\frakM(v_1,w)_{\bbA^1}\times_{\bbA^1}(\M(v_2,\delta_i)_{\bbA^1}\,/\,G(v_2))$$ 
is the 
natural map. As a consequence the map $p : \frakp^{\diamond}_{1,\bbA^1} \to \frakp_{0,\bbA^1}$, 
being the composition of the 
affine fibration $(q_1,p):\frakp^{\diamond}_{1,\bbA^1} \to U'$, the open embedding 
$U'\to  \frakM(v_1,w)_{\bbA^1} \times_{\bbA^1} \frakp_{0,\bbA^1}$ and the projection 
$pr_2: \frakM(v_1,w)_{\bbA^1} \times_{\bbA^1} \frakp_{0,\bbA^1} \to \frakp_{0,\bbA^1}$, is a smooth map.
\end{proof}

\smallskip

We have the following fiber diagram
\begin{align*}
\begin{split}
\xymatrix{
\frakZ_{1,\bbG_m}\ar@{^{(}->}[r]\ar[d]^-p&\,\frakp^\diamond_{1,\bbA^1}\times\frakM(v_2,\delta_i)_{\bbA^1}\ar[d]^-p
&\,\frakp^\diamond_{1}\times\frakM(v_2,\delta_i)\ar[d]^-p\ar@{_{(}->}[l],\\
\frakZ_{0,\bbG_m}\ar@{^{(}->}[r]&\,\frakp_{0,\bbA^1}\times\frakM(v_2,\delta_i)_{\bbA^1}
&\,\frakp_{0}\times\frakM(v_2,\delta_i)\ar@{_{(}->}[l]
}
\end{split}
\end{align*} 
where the vertical maps are smooth and the left inclusions are given by \eqref{inclusion}.
Write 
$$\stab_\epsilon=\stab[v_\epsilon,v_2;w,\delta_i],
\quad
\bfr_\epsilon=\bfr\,[v,0\,;\,v_\epsilon,v_2\,;\,w,\delta_i].$$
Consider the cycle 
\begin{align*}\stab_\varepsilon=\pm\lim_{0}\,[\frakZ_{\epsilon,\bbG_m}]\end{align*} 
which is supported on the set
$$\frakep_\epsilon=\frakp_{\epsilon}\times_{\frakM_0(v,\delta_i)}\frakM(v,\delta_i).$$ 
From Lemma \ref{lem:spe}
we deduce that the restriction of $\stab_1$
to the open subset
$$\frakep_1^\diamond:=\frakp^\diamond_{1}\times_{\frakM_0(v_2,\delta_i)}\frakM(v_2,\delta_i)$$ 
is equal to 
$$\stab_1|_{\frakep_1^\diamond}=p^*(\stab_0).$$
Taking the residue along the $A$-fixed points locus
$$\frakM(v,w)\times\frakM(v_\epsilon,w)\times\frakM(v_2,\delta_i)\subseteq
\frakM(v,w+\delta_i)\times\frakM(v_\epsilon,w)\times\frakM(v_2,\delta_i)$$
of the cycle $\stab_\epsilon$,
we get the cycle
$\bfr_\epsilon$
which is supported on $\frakeh_\epsilon$. 
Using \eqref{hp} we define
$$\frakeh^\diamond_1=\frakeh_1\cap\frakep_1^\diamond.$$ 
Then, we deduce that the restriction of $\bfr_1$ to $\frakeh^\diamond_1$ is equal to $p^*(\bfr_0).$

\smallskip

Now, let $m\in F_{i}(v_2)\otimes K$ be as above. Then, we get
\begin{align*}\e(m,v,w)&=\sum_{\nu\,\vDash\, v_2}a_\nu(v,w)\,[\,\frakh_1[v_1,\Lambda_\nu\,;\,w]\,],\\
\quad
\e(m,v_2,w)&=\sum_{\nu\,\vDash\, v_2}a_\nu(v_2,w)\,[\,\frakh_0[v_0,\Lambda_\nu\,;\,w]\,].
\end{align*}
Therefore, the restriction of $\e(m,v_2,w)$ to 
$$\frakh^\diamond_1:=\frakh_1\cap\frakp_1^\diamond,$$ 
see \eqref{hp}, is equal to $p^*(\e(m,v_2,w))$. 
Since $$\frakh^\diamond_1\cap\frakh_1[v_1,\Lambda_\nu\,;\,w]\neq\emptyset$$ for each composition $\nu$, 
claim \eqref{cas2} is proved.
This finishes the proof of Proposition \ref{lem:5.15}.

\newpage

\end{document}